\documentclass{my_degruyter-proceedings}          %pdflatex

\title{Explicit local time-stepping methods for time-dependent wave propagation}
\headlinetitle{Explicit local time-stepping methods}

\lastnameone{Grote}
\firstnameone{Marcus J.}
\nameshortone{M.\,J.~Grote}
\addressone{Institute of Mathematics, University of Basel, Rheinsprung 21, 4051 Basel}
\countryone{Switzerland}
\emailone{marcus.grote@unibas.ch}

\lastnametwo{Mitkova}
\firstnametwo{Teodora}
\nameshorttwo{T.~Mitkova}
\addresstwo{Department of Mathematics, University of Fribourg, 
Chemin du Mus\'{e}e 23, 1700 Fribourg}
\countrytwo{Switzerland}
\emailtwo{teodora.mitkova@unifr.ch}

\abstract{Semi-discrete Galerkin formulations of transient wave equations, either with
conforming or discontinuous Galerkin finite element discretizations, 
typically lead to large systems of
ordinary differential equations. When explicit time integration is used, the time-step 
is constrained by the smallest elements in the mesh for numerical stability, possibly a high
price to pay. 
To overcome that overly restrictive stability constraint
on the time-step, yet without resorting to implicit methods, explicit 
local time-stepping schemes (LTS) are presented here for transient wave equations
either with or without damping. In the undamped case, leap-frog based LTS methods lead to high-order
explicit LTS schemes, which conserve the energy. In the damped case, when energy is no longer conserved, 
Adams-Bashforth based LTS methods also lead to explicit LTS schemes of 
arbitrarily high accuracy. When combined with a finite element discretization in space with an 
essentially diagonal mass matrix, the resulting time-marching schemes 
are fully explicit and thus inherently parallel. 
Numerical experiments with continuous and discontinuous Galerkin finite element discretizations validate the theory and illustrate
the usefulness of these local time-stepping methods.}



\researchsupported{This work was partly supported by the Swiss National Science Foundation.}

\acknowledgments{We thank Manuela Utzinger and Max Rietmann for their help with the numerical experiments.}

\firstpage{1}

\newcommand{\jmp}[1]{[\![#1]\!]}              % jump
\newcommand{\mvl}[1]{\{\!\!\{#1\}\!\!\}}      % mean value

\begin{document}

%%%%%%%%% SECTION 1 %%%%%%%%%%%%%%%%%%%%%%%%%%%%%%%%%%%%%%%%%%%%%%%%%%%%%%%%%%%%%%%%%%%%%%%%%%%%%%%%%%%%%%%%%%%%%%%%%%%%%
\section{Introduction}
%%%%%%%%%%%%%%%%%%%%%%%%%%%%%%%%%%%%%%%%%%%%%%%%%%%%%%%%%%%%%%%%%%%%%%%%%%%%%%%%%%%%%%%%%%%%%%%%%%%%%%%%%%%%%%%%%%%%%%%%%
The efficient numerical simulation of transient wave phenomena is of fundamental importance in a wide range of applications from acoustics, 
electromagnetics or elasticity. Although classical finite difference methods remain a viable approach in rectangular geometry on Cartesian meshes,
their usefulness is quite limited in the presence of complex geometry, such as cracks, sharp corners or irregular material interfaces. In contrast, 
finite element methods (FEMs) easily handle unstructured meshes and local refinement. Moreover, their extension to high order is straightforward, 
a key feature to keep numerical dispersion minimal.

Semi-discrete finite element Galerkin approximations typically 
lead to a system of ordinary differential equations. However, if explicit time-stepping is
subsequently employed, the mass matrix arising from the spatial discretization by standard conforming finite elements must be inverted at each time-step: 
a major drawback in terms of efficiency. To overcome that difficulty, various ``mass lumping'' techniques have been proposed, which effectively replace 
the mass matrix by a diagonal approximation. While straightforward for piecewise linear elements \cite{GrMi_Ciar78,GrMi_Hug00}, mass lumping techniques require 
special quadrature rules and additional degrees of freedom at higher order to preserve the accuracy and guarantee numerical stability 
\cite{GrMi_CJRT01,GrMi_GT06}.

Discontinuous Galerkin (DG) methods offer an attractive and increasingly popular alternative for the spatial discretization of time-dependent hyperbolic 
problems \cite{GrMi_AMM06,GrMi_CS89,GrMi_GSS06,GrMi_GSS07,GrMi_HW08,GrMi_RW03}. Not only do they accommodate elements of various types and shapes, 
irregular non-matching grids, and even locally varying polynomial order, and hence offer greater flexibility in the mesh design. They also lead to a 
block-diagonal mass matrix, with block size equal to the number of degrees of freedom per element; in fact, for a judicious choice of (locally orthogonal) 
shape functions, the mass matrix is truly diagonal. Thus, when a spatial DG discretization is combined with explicit time integration, the resulting 
time-marching scheme will be truly explicit and inherently parallel.

In the presence of complex geometry, adaptivity and mesh refinement are certainly key for the efficient numerical simulation of wave phenomena. 
However, locally refined meshes impose severe stability constraints on explicit time-marching schemes, where the maximal time-step allowed by the 
CFL condition is dictated by the smallest elements in the mesh. When mesh refinement is restricted to a small region, the use of implicit methods, 
or a very small time-step in the entire computational domain, are a very high price to pay. To overcome this overly restrictive stability constraint, 
various local time-stepping (LTS) schemes \cite{GrMi_CFJ03, GrMi_CFJ06, GrMi_DFFL10} were developed, which use either implicit time-stepping or explicit smaller 
time-steps, but only where the smallest elements in the mesh are located.

Since DG methods are inherently local, they are particularly well-suited for the development of explicit local time-stepping schemes \cite{GrMi_HW08}. 
By combining the sympletic St\"ormer-Verlet method with a DG discretization, Piperno derived a symplectic LTS scheme for Maxwell's equations in a 
non-conducting medium \cite{GrMi_Pip06}, which is explicit and second-order accurate. In \cite{GrMi_MPFC08}, Montseny et al. combined a similar recursive 
integrator with discontinuous hexahedral elements. Starting from the so-called arbitrary high-order derivatives (ADER) DG approach, alternative explicit 
LTS methods for Maxwell's equations \cite{GrMi_TDMS09} and for elastic wave equations \cite{GrMi_DKT07} were proposed. In \cite{GrMi_EJ10}, the LTS approach 
from Collino et al. \cite{GrMi_CFJ03,GrMi_CFJ06} was combined with a DG-FE discretization for the numerical solution of symmetric first-order 
hyperbolic systems. Based on energy conservation, that LTS approach is second-order and explicit inside the coarse and the fine mesh; at the interface, 
however, it nonetheless requires at every time-step the solution of a linear system. More recently, Constantinescu and Sandu devised multirate explicit 
methods for hyperbolic conservation laws, which are based on both Runge-Kutta and Adams-Bashforth schemes combined with a finite volume discretization 
\cite{GrMi_CS07, GrMi_CS09}. Again these multirate schemes are limited to second-order accuracy.

Starting from the standard leap-frog method, Diaz and Grote proposed energy conserving fully explicit LTS integrators of arbitrarily high accuracy for the 
classical wave equation \cite{GrMi_DG09}; that approach was extended to Maxwell's equations in \cite{GrMi_GM10} for non-conductive media. By blending the 
leap-frog and the Crank-Nicolson methods, a second-order LTS scheme was also derived there for (damped) electromagnetic waves in conducting media, 
yet this approach cannot be readily extended beyond order two.
To achieve arbitrarily high accuracy in the presence of dissipation, while remaining fully explicit, 
explicit LTS methods for 
damped wave equations based on Adams-Bashforth (AB) multi-step schemes
were proposed in \cite{GrMi_GM11} -- see also \cite{GrMi_HWW08}. They can also be interpreted as particular
approximations of exponential-Adams multistep methods \cite{GrMi_HO11}.

The rest of the paper is organized as follows. In Section 2, we first recall 
the standard continuous, the symmetric interior penalty (IP) DG and the nodal DG formulations. Next in Section 3, we consider leap-frog based LTS methods, 
both for the undamped and the damped wave equation. In the undamped case, we show how to derive explicit LTS methods of arbitrarily high order; these methods 
also conserve a discrete version of the energy. In the damped case, we present a second-order LTS method by blending the leap-frog and the Crank-Nicolson 
scheme; however, this approach does not easily extend to higher order. To achieve arbitrarily high accuracy even in the presence of dissipation, we then 
consider LTS methods based on Adams-Bashforth multi-step schemes in Section 4. 
Finally in Section 5, we present numerical experiments in one and two space 
dimensions, which validate the theory and underpin both the stability properties and the usefulness of these high-order explicit LTS schemes.

%%%%%%%%% SECTION 2 %%%%%%%%%%%%%%%%%%%%%%%%%%%%%%%%%%%%%%%%%%%%%%%%%%%%%%%%%%%%%%%%%%%%%%%%%%%%%%%%%%%%%%%%%%%%%%%%%%%%%
\section{Finite element discretizations for the wave equation} \label{sec:GrMi_fem_wave}
We consider the damped wave equation
\begin{equation} \label{eq:GrMi_model_eq_1} \begin{split}
u_{tt} + \sigma u_t - \nabla \cdot (c^2 \nabla u) & = f \quad \ \mbox{in } \ \Omega \times (0, T) \,,\\
u(\cdot, t) & = 0 \quad \ \mbox{on } \ \partial \Omega \times (0, T)\,, \\
u(\cdot, 0) = u_0\,, \  u_t(\cdot, 0) & = v_0 \quad \mbox{in } \Omega \,, \\
\end{split} \end{equation}
where $\Omega$ is a bounded Lipschitz domain in ${\mathbb R}^d$, $d = 1, 2, 3$. Here, $f\in L^2(0,T; L^2(\Omega))$ is a (known) source term,
while $u_0\in H^1_0(\Omega)$ and $v_0\in L^2(\Omega)$ are prescribed initial conditions. At the boundary, $\partial \Omega$, we impose
a homogeneous Dirichlet boundary condition, for simplicity. We assume that the damping coefficient, $\sigma = \sigma(x)$ and 
the speed of propagation $c = c(x)$ are piecewise smooth and satisfy the bounds
$$
0 \leq \sigma(x) \leq \sigma^* < \infty\,, \quad 0 < c_* \leq c(x) \leq c^* < \infty\,, \quad x \in \overline{\Omega}\,.
$$

We shall now discretize (\ref{eq:GrMi_model_eq_1}) in space by using any one of the following three distinct FE discretizations:
continuous ($H^1$-conforming) finite elements with mass lumping, a symmetric IP-DG
discretization, or a nodal DG method. Thus, we consider shape-regular
meshes ${\mathcal T}_h$ that partition the domain $\Omega$ into disjoint elements $K$, such that
$\overline \Omega = \cup_{K \in {\mathcal T}_h} \overline K$. The elements are triangles or quadrilaterals in two space dimensions,
and tetrahedra or hexahedra in three dimensions, respectively. The diameter of element $K$ is denoted by $h_K$ and the mesh size,
$h$, is given by $h = \max_{K \in {\mathcal T}_h} h_K$.

%%%%%%%%%%%%%%%%%%%%%%%%%%%%%%%%%%%%%%%%%%%%%%%%%%%%%%%%%%%%%%%%%%%%%%%%%%%%%%%%%%%%%%%%%%%%%%%%%%%%%%%%%%%%%%%%%%%%%%%%%%
\subsection{Continuous Galerkin formulation}
The continuous ($H^1$-conforming) Galerkin formulation of (\ref{eq:GrMi_model_eq_1}) starts from its weak formulation:
find
$u \in [0, T] \to H^1_0(\Omega)$ such that
\begin{equation} \label{eq:GrMi_weak_eq} \begin{split}
(u_{tt}, \varphi) + (\sigma u_t, \varphi) + (c \nabla u, c \nabla \varphi) & = (f, \varphi)
\quad \forall \ \varphi \in H_0^1(\Omega)\,, \quad t \in (0, T) \,, \\
u(\cdot, 0) = u_0 \,, \ u_t(\cdot, 0) & = v_0 \,,
\end{split} \end{equation}
where $(\cdot, \cdot)$ denotes the standard $L^2$ inner product over $\Omega$. It is well-known that (\ref{eq:GrMi_weak_eq}) is well-posed and has a unique
solution \cite{GrMi_LM72}.

For a given partition ${\mathcal T}_h$ of $\Omega$, assumed polygonal (2d) or polyhedral (3d) for simplicity, and an approximation order $\ell \geq 1$, 
we shall approximate the solution $u(\cdot, t)$ of (\ref{eq:GrMi_weak_eq}) in the finite element space
$$
V^h := \left \{ \varphi \in H_0^1(\Omega) \ : \ \varphi|_{K} \in {\mathcal S}^{\ell}(K) \ \
\forall \ K \in {\mathcal T}_h \right \}\,,
$$
where ${\cal S}^{\ell}(K)$ is the space ${\cal P}^{\ell}(K)$ of polynomials of total degree at most $\ell$ on $K$ if $K$ is a
triangle or a tetrahedra, or the space ${\cal Q}^{\ell}(K)$ of polynomials of degree at most $\ell$ in each variable on $K$ if $K$
is a parallelogram or a parallelepiped. Here, we consider the following semi-discrete Galerkin approximation of
(\ref{eq:GrMi_weak_eq}): find $u^h : [0, T] \to V^h$ such that
\begin{equation} \label{eq:GrMi_fe_eq} \begin{split}
(u_{tt}^h, \varphi) + (\sigma u_{t}^h, \varphi) + (c \nabla u^h, c \nabla \varphi) & = (f, \varphi)
\quad \forall \, \varphi \in V^h\,, \quad t \in (0, T) \,, \\
u^h(\cdot, 0) = {\Pi}_h u_0 \,, \ u^h_t(\cdot, 0) &= {\Pi}_h v_0 \,.
\end{split} \end{equation}
Here, ${\Pi}_h$ denotes the $L^2$-projection onto $V^h$.

The semi-discrete formulation (\ref{eq:GrMi_fe_eq}) is equivalent to the second-order system of ordinary differential equations
\begin{equation} \label{eq:GrMi_sdisc_cfem} \begin{split}
{\mathbf M} \, \frac{d^2 {\mathbf U}}{d t^2}(t) + {\mathbf M}_{\sigma} \, \frac{d {\mathbf U}}{d t}(t) +
{\mathbf K} \, {\mathbf U}(t) & = {\mathbf F}(t)\,, \qquad t \in (0, T)\,, \\
{\mathbf M} \, {\mathbf U}(0) = u_0^h \,, \qquad {\mathbf M} \, \frac{d \mathbf U}{d t}(0) & = v_0^h \,.
\end{split} \end{equation}
Here, ${\mathbf U}$ denotes the vector whose components are the time-dependent coefficients of the representation of $u^h$ with respect to 
the finite element nodal basis of $V_h$, $\mathbf M$ the mass matrix, $\mathbf K$ the stiffness matrix, whereas  ${\mathbf M}_{\sigma}$ denotes
the mass matrix  with weight $\sigma$. The matrix ${\mathbf M}$ is sparse, symmetric and positive
definite, whereas the matrices $\mathbf K$ and ${\mathbf M}_{\sigma}$ are sparse, symmetric and, in general, only positive semi-definite. 
In fact, $\mathbf K$ is positive definite, unless Neumann boundary conditions would be imposed
in (\ref{eq:GrMi_model_eq_1}) instead. Since we shall never need to invert $\mathbf K$, our derivation also applies to
the semi-definite case with purely Neumann boundary conditions.

Usually, the mass matrix ${\mathbf M}$ is not diagonal, yet needs to be inverted at every time-step of any explicit time integration scheme.
To overcome this diffculty, various mass lumping techniques have been developed \cite{GrMi_CJRT01, GrMi_Coh02, GrMi_CJT93, GrMi_CJT94},
which essentially replace $\mathbf M$ with a diagonal approximation by computing the required integrals over each element $K$ with judicious quadrature
rules that do not effect the spatial accuracy \cite{GrMi_BD76}.

%%%%%%%%%%%%%%%%%%%%%%%%%%%%%%%%%%%%%%%%%%%%%%%%%%%%%%%%%%%%%%%%%%%%%%%%%%%%%%%%%%%%%%%%%%%%%%%%%%%%%%%%%%%%%%%%%%%%%%%%%%
\subsection{Interior penalty discontinuous Galerkin formulation}
Following \cite{GrMi_GSS06} we briefly recall the symmetric interior penalty (IP) DG formulation of (\ref{eq:GrMi_model_eq_1}). For simplicity,
we assume in this section that the elements are triangles or parallelograms in two space dimensions and tetrahedra or parallelepipeds
in three dimensions, respectively. Generally, we allow for irregular ($k$-irregular) meshes with hanging nodes \cite{GrMi_BPW09}. We denote
by ${\mathcal E}^{\mathcal I}_h$ the set of all interior edges of ${\mathcal T}_h$, by ${\mathcal E}^{\mathcal B}_h$ the set of all
boundary edges of ${\mathcal T}_h$, and set ${\mathcal E}_h = {\mathcal E}^{\mathcal I}_h \cup {\mathcal E}^{\mathcal B}_h$. Here, we
generically refer to any element of ${\mathcal E}_h$ as an ``edge'', that is a real edge in 2d and a face in 3d.

For a piecewise smooth function $\varphi$, we introduce the following trace operators. Let $e \in {\mathcal E}^{\mathcal I}_h$ be an
interior edge shared by two elements $K^+$ and $K^-$ with unit outward normal vectors ${\mathbf n}^{\pm}$, respectively. Denoting
by $v^{\pm}$ the trace of $v$ on $\partial K^\pm$ taken from within $K^\pm$, we define the jump and the average on $e$ by
$$
\jmp{\varphi} := \varphi^+ {\mathbf n}^+ + \varphi^-{\mathbf n}^-  \,, \qquad
\mvl{\varphi}:= (\varphi^+ + \varphi^-)/2\,.
$$
On every boundary edge $e \in {\mathcal E}^{\mathcal B}_h$, we set $\jmp{\varphi} := \varphi {\mathbf n}$ and
$\mvl{\varphi} := \varphi$. Here, $\mathbf{n}$ is the outward unit normal to the domain boundary $\partial \Omega$.

For a piecewise smooth vector-valued function ${\psi}$, we analogously define the average across interior faces by
$\mvl{{\psi}}:=({\psi}^+ + {\psi}^-)/2$, and on boundary faces we set $\mvl{{\psi}}:={\psi}$. The jump of a
vector-valued function will not be used.  For a vector-valued function ${\psi}$ with continuous normal components across a face $e \in {\mathcal E}_h$, the trace identity
\begin{equation*}
\varphi^+\left( \mathbf{n}^{+} \cdot {\psi}^+ \right) + \varphi^-\left( \mathbf{n}^{-} \cdot {\psi}^- \right)
 = \jmp{\varphi} \cdot \mvl{{\psi}} \quad \mbox{on } e\,,
\end{equation*}
immediately follows from the above definitions.

For a given partition ${\cal T}_h$ of $\Omega$ and an approximation order $\ell \geq 1$, we wish to approximate the solution
$u(t, \cdot)$ of (\ref{eq:GrMi_model_eq_1}) in the finite element space
\begin{equation*}
V^h :=  \left \{ \varphi \in L^2(\Omega): \, \varphi|_K\in {\cal S}^{\ell}(K) ~~ \forall K\in {\cal T}_h \right \} \,,
\end{equation*}
where ${\mathcal S}^{\ell}(K)$ ist the space ${\mathcal P}^{\ell}(K)$ (for triangles or tetrahedra) or
${\mathcal Q}^{\ell}(K)$ (for quadrilaterals or hexahedra). Thus, we consider the following (semidiscrete) DG approximation of (\ref{eq:GrMi_model_eq_1}): 
find
$u^h : [0, T] \to V^h$ such that
\begin{equation} \label{eq:GrMi_dg_eq} \begin{split}
(u_{tt}^h, \varphi) + (\sigma u_{t}^h, \varphi) + a_h(u^h, \varphi) & = (f,\varphi)
\qquad \forall \, \varphi \in V^h\,, \quad t \in (0, T) \,, \\
u^h(\cdot, 0) = {\Pi}_h u_0 \,, \ u^h_t(\cdot, 0) &= {\Pi}_h v_0 \,.
\end{split} \end{equation}
Here, ${\Pi}_h$ again denotes the $L^2$-projection onto $V^h$ whereas the DG bilinear form $a_h(\cdot, \cdot)$, defined on $V^h \times V^h$, is given by
\begin{equation} \label{eq:GrMi_dg_form} \begin{split}
a_{h}(u, \varphi) :=  & \sum_{K \in {\mathcal T}_h} \int_K c^2 \, \nabla u \cdot \nabla \varphi \,dx -
\sum_{e \in {\mathcal E}_h} \int_{e} \jmp{u} \cdot \mvl{c^2 \, \nabla \varphi} \, dA \\
& -\sum_{e \in{\mathcal E}_h} \int_{e} \jmp{\varphi} \cdot \mvl{c^2 \, \nabla u} \, dA +
\sum_{e \in {\mathcal E}_h} \int_{e}\, {\tt a} \, \jmp{u} \cdot \jmp{\varphi} \,dA \, .
\end{split} \end{equation}
The last three terms in (\ref{eq:GrMi_dg_form}) correspond to jump and flux terms at element boundaries; they vanish when
$u, \varphi, \in H^1_0(\Omega)\cap H^{1+m}(\Omega)$ for $m > \frac{1}{2}$. Hence, the above semi-discrete DG formulation (\ref{eq:GrMi_dg_eq})
is consistent with the original continuous problem~(\ref{eq:GrMi_weak_eq}).

In (\ref{eq:GrMi_dg_form}) the function ${\tt a}$ penalizes the jumps of $u$ and $v$ over the faces of ${\mathcal T}_h$. To define it, we first
introduce the functions $\tt h$ and $\tt c$ by
{\small $$
{\tt h} |_e = \left \{ \begin{array}{ll}
\min \{ h_{K^+}, h_{K^-} \}, & e \in {\mathcal E}_h^{\mathcal I}\,,\\[5pt] h_K, & e \in {\mathcal E}_h^{\mathcal B}\,,
\end{array} \right. \
{\tt c} |_e(x) = \left \{ \begin{array}{ll}
\max \{ c|_{K^+}(x), c|_{K^-}(x) \}, & e \in {\mathcal E}_h^{\mathcal I} \,,\\[5pt] c|_K(x), & e \in {\mathcal E}_h^{\mathcal B} \,.
\end{array} \right.
$$}
Then, on each $e \in {\mathcal E}_h$, we set
\begin{equation} \label{eq:GrMi_param_alpha}
{\tt a} |_e := \alpha \, {\tt c}^2 {\tt h}^{-1}\,,
\end{equation}
where $\alpha$ is a positive parameter independent of the local mesh sizes and the coefficient $c$. There exists a threshold value
$\alpha_{min} > 0$, which depends only on the shape regularity of the mesh  and the approximation order $\ell$ such that for 
$\alpha \geq \alpha_{min}$ the DG bilinear form $a_h$ is coercive and, hence, the discretization is stable \cite{GrMi_ABCM01, GrMi_AD10}. Throughout the
rest of the paper we shall assume that $\alpha \geq \alpha_{min}$ so that the semi-discrete problem (\ref{eq:GrMi_dg_eq}) has
a unique solution which converges with optimal order \cite{GrMi_GSS06,GrMi_GSS07,GrMi_GSS08, GrMi_GS09}. In \cite{GrMi_GSS06,GrMi_GS09}, a detailed convergence analysis
and numerical study of the IP-DG method for (\ref{eq:GrMi_dg_form}) with $\sigma = 0$ was presented. In particular, optimal a-priori estimates
in a DG-energy norm and the $L^2$-norm were derived. This theory immediately generalizes to the case $\sigma \geq 0$. For sufficiently
smooth solutions, the IP-DG method (\ref{eq:GrMi_dg_form}) thus yields the optimal $L^2$-error estimate of order ${\mathcal O}(h^{\ell + 1})$.

The semi-discrete IP-DG formulation (\ref{eq:GrMi_dg_eq}) is equivalent to the second-order system of ordinary differential equations (\ref{eq:GrMi_sdisc_cfem}).
Again, the mass matrix ${\mathbf M}$ is sparse, symmetric and positive definite. Yet because individual elements decouple,
${\mathbf M}$ (and ${\mathbf M}_{\sigma}$) is block-diagonal with block size equal to the number of degrees of freedom per element.
Thus, ${\mathbf M}$ can be inverted at very low computational cost. In fact, for a judicious choice of (locally orthogonal) shape functions,
${\mathbf M}$ is truly diagonal.

%%%%%%%%%%%%%%%%%%%%%%%%%%%%%%%%%%%%%%%%%%%%%%%%%%%%%%%%%%%%%%%%%%%%%%%%%%%%%%%%%%%%%%%%%%%%%%%%%%%%%%%%%%%%%%%%%%%%%%%%%%
\subsection{Nodal discontinuous Galerkin formulation}
Finally, we briefly recall the nodal discontinuous Galerkin formulation from \cite{GrMi_HW08} for the spatial discretization of (\ref{eq:GrMi_model_eq_1})
rewritten as a first-order system. To do so, we first let $v := u_t$, ${\mathbf w}:=-\nabla u$, and thus we rewrite (\ref{eq:GrMi_model_eq_1}) as the
first-order hyperbolic system:
\begin{equation} \label{eq:GrMi_model_eq_tmp} \begin{split}
v_t + \sigma v + \nabla \cdot (c^2 {\mathbf w}) & = f \quad \quad \quad \mbox{in } \ \Omega \times (0, T) \,,\\
{\mathbf w}_t + \nabla v & = {\mathbf 0} \quad \quad \quad \mbox{in } \ \Omega \times (0, T) \,,\\
v(\cdot, t) = 0\,, \  {\mathbf w}(\cdot, t) & = {\mathbf 0} \quad \quad \quad  \mbox{on } \ \partial \Omega \times (0, T)\,, \\
v(\cdot, 0) = v_0\,, \ {\mathbf w}(\cdot, 0) & = - \nabla u_0  \quad \mbox{in } \Omega \,,
\end{split} \end{equation}
or in more compact notation as
\begin{equation} \label{eq:GrMi_model_eq_2}
{\mathbf q}_t + {\mathbf \Sigma} \, {\mathbf q} + \nabla \cdot {\mathcal F}({\mathbf q}) = {\mathbf S} \,,
\end{equation}
with
\begin{equation*}
{\mathbf q} = \left (\begin{array}{c} v \\ {\mathbf w}\end{array} \right), \,
 {\mathcal F}({\mathbf q}) =    \left (\begin{array}{c} c^2\, {\mathbf w}^{\top} \\ v\, {\mathbf I}_{d\times d}\end{array} \right),\, {\mathbf\Sigma} = \left ( \begin{array}{rr}
{\mathbf \sigma} & {\mathbf 0} \\ {\mathbf 0} & {\mathbf 0} \end{array} \right ), \, {\mathbf S} =   \left (\begin{array}{c} f \\ {\mathbf 0}\end{array} \right).
\end{equation*}
Following \cite{GrMi_HW08}, we now consider the following nodal DG formulation of (\ref{eq:GrMi_model_eq_2}):
find ${\mathbf q}^h : [0, T] \to {\mathbf V}^h$ such that
\begin{equation} \label{eq:GrMi_dg1_pr}
({\mathbf q}_{t}^h, {\psi}) + ({\mathbf \Sigma} \, {\mathbf q}^h, {\psi}) +
{\widetilde a}_h({\mathbf q}^h, {\psi}) = ({\mathbf S}, {\psi})
\qquad \forall \, {\psi} \in {\mathbf V}^h\,, \quad t \in (0, T) \,.
\end{equation}
Here ${\mathbf V}^h$ denotes the finite element space
\begin{equation*}
{\mathbf  V}^h :=  \left \{ {\psi} \in L^2(\Omega)^{d+1}: \, {\psi}|_K\in {\cal S}^{\ell}(K)^{d+1} ~~ \forall K\in {\cal T}_h \right \}
\end{equation*}
for a given partition ${\cal T}_h$ of $\Omega$ and an approximation order $\ell \geq 1$. The nodal-DG bilinear form ${\widetilde a}_h(\cdot , \cdot)$ is defined on
${\mathbf V}^h \times {\mathbf V}^h$ as
$$
{\widetilde a}_h({\mathbf q}, {\psi}) :=
\sum_{K \in {\cal T}_h} \int_K \left ( \nabla \cdot {\mathcal F}({\mathbf q}) \right ) \cdot {\psi} \, dx -
\sum_{e \in {\mathcal E}_h} \int_e \left ( {\mathbf n} \cdot {\mathcal F}({\mathbf q}) -
({\mathbf n} \cdot {\mathcal F}({\mathbf q}))^* \right ) \cdot {\psi} \, dA\,,
$$
where $({\mathbf n} \cdot {\mathcal F}({\mathbf q}))^*$ is a suitably chosen numerical flux in the unit normal direction $\mathbf n$. The
semi-discrete problem (\ref{eq:GrMi_dg1_pr}) has a unique solution, which converges with optimal order in the $L^2$-norm \cite{GrMi_HW08}.

The semi-discrete nodal DG formulation (\ref{eq:GrMi_dg1_pr}) is equivalent to the first-order system of ordinary differential equations
\begin{equation} \label{eq:GrMi_sd_pr2}
{\mathbf M} \frac{d {\mathbf Q}}{dt}(t) + {\mathbf M}_\sigma \, {\mathbf Q}(t) + {\mathbf C} \, {\mathbf Q}(t) = {\mathbf F}(t) \,,
\qquad t \in (0, T)\,.
\end{equation}
Here ${\mathbf Q}$ denotes the vector whose components are the coefficients of ${\mathbf q}^h$ with respect to the finite element basis of
${\mathbf V}^h$ and $\mathbf C$ the DG stiffness matrix. Because the individual elements decouple, the mass matrices ${\mathbf M}$ and
${\mathbf M}_{\sigma}$ are sparse, symmetric, positive semi-definite and block-diagonal. Moreover, ${\mathbf M}$ is positive definite and can be inverted at
very low computational cost.
%%%%%%%%%%%%%%%%%%%%%%%%%%%%%%%%%%%%%%%%%%%%%%%%%%%%%%%%%%%%%%%%%%%%%%%%%%%%%%%%%%%%%%%%%%%%%%%%%%%%%%%%%%%%%%%%%%%%%%%%%

%%%%%%%%% SECTION 3 %%%%%%%%%%%%%%%%%%%%%%%%%%%%%%%%%%%%%%%%%%%%%%%%%%%%%%%%%%%%%%%%%%%%%%%%%%%%%%%%%%%%%%%%%%%%%%%%%%%%%
\section{Leap-frog based LTS methods}\label{sec:GrMi_LTS-LF}
Starting from the well-known second-order ``leap-frog'' scheme, we now derive an explicit second-order LTS scheme for undamped waves. By using the
modified equation approach, we then derive an
explicit fourth-order LTS method for undamped waves. Finally, by blending the leap-frog and the Crank-Nicolson methods,
we also present a second-order LTS scheme for damped waves.

We consider the semi-discrete model equation
\begin{equation}\label{eq:GrMi_semi_disc_model}
{\mathbf M} \, \frac{d^2 {\mathbf U}}{d t^2}(t) + {\mathbf M}_{\sigma} \, \frac{d {\mathbf U}}{d t}(t) +
{\mathbf K} \, {\mathbf U} = {\mathbf F}(t) \,,
\end{equation}
where ${\mathbf M}$ and ${\mathbf M}_{\sigma}$ are symmetric positive definite matrices and $\mathbf K$ is a symmetric positive semi-definite matrix.
Moreover, we assume that the mass matrix ${\mathbf M}$ is (block-) diagonal, as in (\ref{eq:GrMi_sdisc_cfem}).
We remark, however, that the time integration techniques presented below are also applicable to other spatial discretizations of the damped wave equation
that lead to the same semi-discrete form (\ref{eq:GrMi_semi_disc_model}).

Because ${\mathbf M}$ is assumed essentially diagonal, ${\mathbf M}^{\frac{1}{2}}$ can be explicitly computed and inverted at low cost.
Thus, we multiply (\ref{eq:GrMi_semi_disc_model}) by ${\mathbf M}^{-\frac{1}{2}}$ to obtain
\begin{equation} \label{eq:GrMi_semi_disc_model_final}
\frac{d^2 {\mathbf z}}{d t^2}(t) + {\mathbf D} \, \frac{d {\mathbf z}}{d t}(t) + {\mathbf A} \, {\mathbf z}(t) = {\mathbf R}(t) \,,
\end{equation}
with ${\mathbf z} = {\mathbf M}^{\frac{1}{2}} {\mathbf U}$,
${\mathbf D} = {\mathbf M}^{-\frac{1}{2}} {\mathbf M}_{\sigma} {\mathbf M}^{-\frac{1}{2}}$,
${\mathbf A} = {\mathbf M}^{-\frac{1}{2}} {\mathbf K} {\mathbf M}^{-\frac{1}{2}}$ and
${\mathbf R} = {\mathbf M}^{-\frac{1}{2}} {\mathbf F}$. Note that $\mathbf A$ is also sparse and symmetric positive
semidefinite. For undamped waves, ${\mathbf D}$ vanishes and hence energy is conserved, whereas for damped waves ${\mathbf D}$ is nonzero and 
energy is dissipated. We shall distinguish these two situations in the derivation of local time-stepping schemes below.

%%%%%%%%%%%%%%%%%%%%%%%%%%%%%%%%%%%%%%%%%%%%%%%%%%%%%%%%%%%%%%%%%%%%%%%%%%%%%%%%%%%%%%%%%%%%%%%%%%%%%%%%%%%%%%%%%%%%%%%%%
\subsection{Second-order method for undamped waves}\label{sec:GrMi_LTS-LF2}
For undamped waves, (\ref{eq:GrMi_semi_disc_model_final}) reduces to
\begin{equation} \label{eq:GrMi_semi_disc_model_sigma_0}
\frac{d^2 {\mathbf z}}{d t^2} + {\mathbf A} \, {\mathbf z} = {\mathbf R} \,.
\end{equation}
Since for any $f \in {\mathcal C}^2$, we have
\begin{equation} \label{eq:GrMi_exact_formula}
f(t + \Delta t) - 2 \, f(t) + f(t - \Delta t) = \Delta t^2 \, \int_{-1}^{1} (1 - \vert \theta \vert) f''(t + \theta \, \Delta t)
\, d \theta \,,
\end{equation}
the exact solution ${\mathbf z}(t)$ of (\ref{eq:GrMi_semi_disc_model_sigma_0}) satisfies
\begin{equation} \label{eq:GrMi_integral}
{\mathbf z}(t + \Delta t) - 2 \, {\mathbf z}(t) + {\mathbf z}(t - \Delta t) = \Delta t^2 \, \int_{-1}^{1} (1 - \vert \theta \vert)
\left ( {\mathbf R}(t + \theta \, \Delta t) - {\mathbf A}{\mathbf z}(t + \theta \, \Delta t) \right )d \theta \,.
\end{equation}
The integral on the right side of (\ref{eq:GrMi_integral}) represents a weighted average of ${\mathbf R}(s) - {\mathbf A} \,{\mathbf z}(s)$
over the interval $[t - \Delta t, t + \Delta t]$, which needs to be approximated in any numerical algorithm. If we approximate in (\ref{eq:GrMi_integral})
${\mathbf A} \, {\mathbf z}(t + \theta \, \Delta t)$ and ${\mathbf R}(t + \theta \, \Delta t)$ by ${\mathbf A} \, {\mathbf z}(t)$ and
${\mathbf R}(t)$, respectively, and evaluate the remaining $\theta$-dependent integral, we obtain the well-known
second-order leap-frog scheme with time-step $\Delta t$,
\begin{equation} \label{eq:GrMi_standard_leap_frog}
{\mathbf z}_{n+1} - 2 \, {\mathbf z}_{n} + {\mathbf z}_{n-1} = \Delta t^2 \left ( {\mathbf R}_n - {\mathbf A} \, {\mathbf z}_n
\right ), \quad {\mathbf R}_n \simeq {\mathbf R}(t_n), \, {\mathbf z}_n \simeq {\mathbf z}(t_n) \,,
\end{equation}
which, however, would require $\Delta t$ to be comparable in size to the smallest elements in the mesh for numerical stability.

Following \cite{GrMi_DG09, GrMi_HLW02}, we
instead split the vectors ${\mathbf z}(t)$ and ${\mathbf R}(t)$ as
\begin{equation} \label{eq:GrMi_eq_split} \begin{split}
{\mathbf z}(t) & = ({\mathbf I} - {\mathbf P}) {\mathbf z}(t) + {\mathbf P} {\mathbf z}(t) =
{\mathbf z}^{[\mbox{\scriptsize{coarse}}]}(t) + {\mathbf z}^{[\mbox{\scriptsize fine}]}(t)\,, \\
{\mathbf R}(t) & = ({\mathbf I} - {\mathbf P}) {\mathbf R}(t) + {\mathbf P} {\mathbf R}(t) =
{\mathbf R}^{[\mbox{\scriptsize coarse}]}(t) + {\mathbf R}^{[\mbox{\scriptsize fine}]}(t)\,,
\end{split} \end{equation}
where the projection matrix $\mathbf P$ is diagonal. Its diagonal entries, equal to zero or one, identify the unknowns associated with
the locally refined region, where smaller time-steps are needed. To circumvent the severe CFL restriction on $\Delta t$ in the
leap-frog scheme, we need to treat ${\mathbf z}^{[\mbox{\scriptsize fine}]}(t)$ and ${\mathbf R}^{[\mbox{\scriptsize fine}]}(t)$
differently from ${\mathbf z}^{[\mbox{\scriptsize coarse}]}(t)$ and ${\mathbf R}^{[\mbox{\scriptsize coarse}]}(t)$ in
\begin{equation} \label{eq:GrMi_integral2} \begin{split}
& {\mathbf z}(t + \Delta t) - 2 \, {\mathbf z}(t) + {\mathbf z}(t - \Delta t) \\
& = \Delta t^2 \, \int_{-1}^{1} (1 - \vert \theta \vert)  \left \{ {\mathbf R}^{[\mbox{\scriptsize coarse}]}(t + \theta \, \Delta t)
+ {\mathbf R}^{[\mbox{\scriptsize fine}]} (t + \theta \, \Delta t) \right. \\
& \left. \qquad \qquad - {\mathbf A} \left ( {\mathbf z}^{[\mbox{\scriptsize coarse}]}(t + \theta \, \Delta t) +
{\mathbf z}^{[\mbox{\scriptsize fine}]}(t + \theta \, \Delta t) \right ) \right \} \, d \theta \,.
\end{split} \end{equation}
Since we wish to use the standard leap-frog scheme in the coarse part of the mesh, we approximate the terms in (\ref{eq:GrMi_integral2}) that involve
${\mathbf z}^{[\mbox{\scriptsize coarse}]}(t + \theta \, \Delta t)$ and
${\mathbf R}^{[\mbox{\scriptsize coarse}]}(t + \theta \, \Delta t)$ by their values at $t$, which yields
\begin{equation} \label{eq:GrMi_integral3} \begin{split}
{\mathbf z}(t + \Delta t) & - 2 \, {\mathbf z}(t) + {\mathbf z}(t - \Delta t) \simeq
\Delta t^2 \left \{ ({\mathbf I} - {\mathbf P}) {\mathbf R}(t) - {\mathbf A} ({\mathbf I} - {\mathbf P}) {\mathbf z}(t) \right \} \\
& + \Delta t^2 \, \int_{-1}^{1} (1 - \vert \theta \vert) \left \{
{\mathbf P} {\mathbf R}(t + \theta \Delta t) - {\mathbf A} {\mathbf P} {\mathbf z}(t + \theta \Delta t) \right \} \, d \theta \,.
\end{split} \end{equation}
Note that $\mathbf A$ and $\mathbf P$ do not commute.

Next for fixed $t$, let ${\widetilde {\mathbf z}}(\tau)$ solve the differential equation
\begin{equation} \label{eq:GrMi_diff_eq_z} \begin{split}
\frac{d^2 \widetilde{\mathbf z}}{d \tau^2}(\tau) & =
(\mathbf I - \mathbf P) {\mathbf R}(t) - {\mathbf A} (\mathbf I - \mathbf P) {\mathbf z}(t) +
{\mathbf P} {\mathbf R}(t + \tau) - {\mathbf A} {\mathbf P}  \widetilde{\mathbf z}(\tau) \,, \\
\widetilde{\mathbf z}(0) & = {\mathbf z}(t) \,, \ \widetilde{\mathbf z}'(0) = {\nu} \,,
\end{split} \end{equation}
where $\nu$ will be specified below. Again from (\ref{eq:GrMi_exact_formula}), we deduce that
\begin{equation} \label{eq:GrMi_integral4} \begin{split}
\widetilde {\mathbf z}(\Delta t) & - 2 \, \widetilde {\mathbf z}(0) + \widetilde {\mathbf z}(- \Delta t) =
\Delta t^2 \left \{ ({\mathbf I} - {\mathbf P}) {\mathbf R}(t) - {\mathbf A} ({\mathbf I} - {\mathbf P}) {\mathbf z}(t) \right \} \\
& + \Delta t^2 \, \int_{-1}^{1} (1 - \vert \theta \vert) \left \{
{\mathbf P} {\mathbf R}(t + \theta \Delta t) - {\mathbf A} {\mathbf P} \widetilde{\mathbf z}(\theta \Delta t) \right \}
 \, d \theta \,.
\end{split} \end{equation}
From the comparison of (\ref{eq:GrMi_integral3}) with (\ref{eq:GrMi_integral4}), we infer that
$$
{\mathbf z}(t + \Delta t) - 2 \, {\mathbf z}(t) + {\mathbf z}(t - \Delta t) \simeq
\widetilde {\mathbf z}(\Delta t) - 2 \, \widetilde {\mathbf z}(0) + \widetilde {\mathbf z}(- \Delta t) \,,
$$
or equivalently
\begin{equation} \label{eq:GrMi_approx_eq}
{\mathbf z}(t + \Delta t) + {\mathbf z}(t - \Delta t) \simeq 
\widetilde {\mathbf z}(\Delta t) + \widetilde {\mathbf z}(- \Delta t) \,.
\end{equation}
In fact from Taylor expansion and (\ref{eq:GrMi_semi_disc_model_sigma_0}), we obtain
\begin{equation*} \begin{split}
& \widetilde{\mathbf z}(\Delta t) + \widetilde {\mathbf z}(- \Delta t) = 2 \widetilde {\mathbf z}(0) +
\widetilde {\mathbf z}''(0) \Delta t^2 + {\mathcal O}(\Delta t^4) \\
& = 2 {\mathbf z}(t) + ({\mathbf R}(t) - {\mathbf A} {\mathbf z}(t)) \Delta t^2  + {\mathcal O}(\Delta t^4)
= {\mathbf z}(t + \Delta t) + {\mathbf z}(t - \Delta t) + {\mathcal O}(\Delta t^4) \,.
\end{split}
\end{equation*}
Thus to advance ${\mathbf z}(t)$ from $t$ to $t + \Delta t$, we shall evaluate $\widetilde{\mathbf z}(\Delta t) + \widetilde {\mathbf z}(- \Delta t)$
by solving (\ref{eq:GrMi_diff_eq_z}) numerically.

To take advantage of the inherent symmetry in time and thereby reduce the computational effort even further, we now let
$$
\mathbf q(\tau) = \widetilde{\mathbf z}(\tau) + \widetilde {\mathbf z}(- \tau) \,.
$$
Then, $q(\tau)$  solves the differential equation
\begin{equation} \label{eq:GrMi_diff_eq_q} \begin{split}
\frac{d^2 {\mathbf q}}{d \tau^2}(\tau) & = 2 \left \{
(\mathbf I - \mathbf P) {\mathbf R}(t) - {\mathbf A} (\mathbf I - \mathbf P) {\mathbf z}(t) \right \}
+ {\mathbf P} \left \{ {\mathbf R}(t + \tau) + {\mathbf R}(t - \tau) \right \} \\
& \quad - {\mathbf A} {\mathbf P} {\mathbf q}(\tau) \,, \\
{\mathbf q}(0) & = 2 {\mathbf z}(t) \,, \ {\mathbf q}'(0) = 0 \,,
\end{split} \end{equation}
with
\begin{equation} \label{eq:GrMi_approx_eq_q}
{\mathbf z}(t + \Delta t) + {\mathbf z}(t - \Delta t) = {\mathbf q}(\Delta t) + {\mathcal O}(\Delta t^4)\,.
\end{equation}
Note that ${\mathbf q}(\Delta t)$ does not depend on the value of $\nu$. Now, we shall approximate the right side of (\ref{eq:GrMi_integral}) by solving
(\ref{eq:GrMi_diff_eq_q}) on $[0, \Delta t]$, and then use (\ref{eq:GrMi_approx_eq_q}) to compute ${\mathbf z}(t + \Delta t)$. Thus, we need the numerical value of
$\mathbf q(\tau)$ only at $\Delta t$.

In summary, the second-order LTS algorithm for the solution of (\ref{eq:GrMi_semi_disc_model_sigma_0}) computes
${\mathbf z}_{n+1} \simeq \mathbf z(t + \Delta t)$, given $\mathbf z_n$ and $\mathbf z_{n-1}$, as follows:

\bigskip

\noindent {\bf LTS-LF2($p$) Algorithm}{\it
\begin{enumerate}
\item Set ${\mathbf w} := ({\mathbf I} - {\mathbf P}){\mathbf R}_n - {\mathbf A}({\mathbf I} - {\mathbf P}) {\mathbf z}_n$ and
${\mathbf q}_0 := 2 {\mathbf z}_n$.
\item Compute
$
{\mathbf q}_{1/p} := {\mathbf q}_0 + \displaystyle \frac{1}{2} \left ( \frac{\Delta t}{p} \right )^2
\left ( 2 \mathbf w + 2 {\mathbf P} {\mathbf R}_{n, 0} - {\mathbf A} {\mathbf P} {\mathbf q}_0 \right ) \,.
$
\item For $m = 1, \dots, p-1$, compute
{\small \begin{equation*}
{\mathbf q}_{(m+1)/p} := 2 {\mathbf q}_{m/p} - {\mathbf q}_{(m-1)/p} + \left ( \frac{\Delta t}{p} \right )^2 \left ( 2 \mathbf w +
{\mathbf P} ({\mathbf R}_{n, m} + {\mathbf R}_{n, -m}) - {\mathbf A} {\mathbf P} {\mathbf q}_{m/p} \right ) \,.
\end{equation*}}
\item Compute ${\mathbf z}_{n+1} := -{\mathbf z}_{n-1} + \mathbf q_1$.
\end{enumerate}}

Here, we have used the notations ${\mathbf R}_{n, m} \simeq {\mathbf R}(t_n + \tau_m)$ and 
${\mathbf R}_{n, -m} \simeq {\mathbf R}(t_n - \tau_m)$, where $t_n = n \Delta t$ and $\tau_m = m \Delta \tau$; note that
${\mathbf R}_{n, 0} \simeq {\mathbf R}(t_n + \tau_0) = {\mathbf R}(t_n) \simeq {\mathbf R}_n$. Steps 1-3 correspond to
the numerical solution of (\ref{eq:GrMi_diff_eq_q}) until $\tau = \Delta t$ with the leap-frog scheme, using the local time-step
$\Delta \tau = \Delta t /p$. For $\mathbf P = \mathbf 0$, that is without any local time-stepping, we recover the standard
leap-frog scheme. If the fraction of nonzero entries in $\mathbf P$ is small, the overall cost is dominated by the computation of
$\mathbf w$, which requires one multiplication by ${\mathbf A}({\mathbf I} - {\mathbf P})$ per time-step $\Delta t$. All further matrix-vector
multiplications by ${\mathbf A} {\mathbf P}$ only affect those unknowns that lie inside the refined region, or
immediately next to it.

\begin{proposition}
For ${\mathbf R}(t) \in C^2([0, T])$, the local time-stepping method LTS-LF2($p$) is second-order accurate.
\end{proposition}
\begin{proof} See \cite{GrMi_GM10}.\end{proof}

To establish the stability of the LTS-LF2($p$) scheme we consider the homogeneous case, ${\mathbf R}_{n} = {\mathbf 0}$.
Then, the standard leap-frog scheme (\ref{eq:GrMi_standard_leap_frog}) conserves the
discrete energy
{\small \begin{equation} \label{eq:GrMi_disc_energy_LF}
E^{n + \frac{1}{2}} = \frac{1}{2} \left [ \left < \left ( {\mathbf I} - \frac{\Delta t^2}{4}
{\mathbf A} \right ) \frac{{\mathbf z}_{n+1} - {\mathbf z}_n }{\Delta t}, \frac{{\mathbf z}_{n+1} - {\mathbf z}_n}{\Delta t} \right > +
\left < {\mathbf A} \frac{{\mathbf z}_{n+1} + {\mathbf z}_n}{2}, \frac{{\mathbf z}_{n+1} + {\mathbf z}_n}{2} \right > \right ] \,.
\end{equation}}
Here $E^{n + \frac{1}{2}} \simeq E(t_{n+\frac{1}{2}})$ and the angular brackets denote the standard Euclidean inner product.
Since $\mathbf A$ is symmetric, the quadratic form in (\ref{eq:GrMi_disc_energy_LF}) is also symmetric. For sufficiently small $\Delta t$ it is
also positive semidefinite and hence yields a true energy.

To derive a necessary and sufficient condition for the numerical stability of the LTS-LF2($p$) scheme, we exhibit a conserved discrete energy for the
LTS-LF2($p$) algorithm with ${\mathbf R}_n = {\mathbf 0}$.  Following \cite{GrMi_DG09}, we first rewrite the LTS-LF2($p$) scheme in ``leap-frog manner''.
\begin{proposition}
The local time-stepping scheme LTS-LF2($p$) with ${\mathbf R}_{n,m} = {\mathbf 0}$ is equivalent to
$$
{\mathbf z}_{n+1} = 2 \mathbf{z}_n - {\mathbf z}_{n-1} - \Delta t^2 {\mathbf A}_p {\mathbf z}_n\,,
$$
where ${\mathbf A}_p$ is defined by
\begin{equation} \label{eq:GrMi_matrix_Ap}
{\mathbf A}_p = {\mathbf A} - \frac{2}{p^2} \sum_{j = 1}^{p-1} \left ( \frac{\Delta t}{p} \right )^{2j} \alpha_j^p
({\mathbf A} {\mathbf P})^j {\mathbf A}
\end{equation}
and the constants $\alpha_j^p$ are given by
\begin{equation*} \begin{split}
\alpha_1^2 & = 1, \quad \alpha_1^3 = 6, \quad \alpha_2^3 = -1, \\
\alpha_1^{p+1} & = m^2 + 2 \alpha_1^p - \alpha_1^{p-1}, \\
\alpha_j^{p+1} & = 2 \alpha_j^p - \alpha_j^{p-1} - \alpha_{j-1}^p, \quad j = 2, \dots, p-2, \\
\alpha_{p-1}^{p+1} & = 2 \alpha_{p-1}^{p} - \alpha_{p-2}^p, \\
\alpha_p^{p+1} & = - \alpha_{p-1}^p\,.
\end{split} \end{equation*}
Furthermore, the matrix ${\mathbf A}_p$ is symmetric.
\end{proposition}
\begin{proof} See \cite{GrMi_DG09} and \cite{GrMi_GM10}.\end{proof}

\begin{proposition}
The local time-stepping scheme LTS-LF2($p$) with ${\mathbf R}_n = {\mathbf 0}$ conserves the energy
{\small \begin{equation}\label{eq:GrMi_disc_energy_LT}
E^{n + \frac{1}{2}} = \frac{1}{2} \left [ \left < \left ( {\mathbf I} - \frac{\Delta t^2}{4}
{\mathbf A}_p \right ) \frac{{\mathbf z}_{n+1} - {\mathbf z}_n }{\Delta t}, \frac{{\mathbf z}_{n+1} - {\mathbf z}_n}{\Delta t} \right > +
\left < {\mathbf A}_p \frac{{\mathbf z}_{n+1} + {\mathbf z}_n}{2}, \frac{{\mathbf z}_{n+1} + {\mathbf z}_n}{2} \right > \right ] \,.
\end{equation}}
\end{proposition}
\begin{proof}
By symmetry of ${\mathbf A}_p$, this standard argument is similar the proof of (\ref{eq:GrMi_disc_energy_LF}); see also \cite{GrMi_DG09} for details.
\end{proof}
As a consequence, the LTS-LF2($p$) is stable if $0<(\Delta t^2 / 4)\lambda_{\max}({\mathbf A}_p) < 1$;  note 
that the matrix ${\mathbf A}_p$ itself also depends on $\Delta t$.

%%%%%%%%%%%%%%%%%%%%%%%%%%%%%%%%%%%%%%%%%%%%%%%%%%%%%%%%%%%%%%%%%%%%%%%%%%%%%%%%%%%%%%%%%%%%%%%%%%%%%%%%%%%%%%%%%%%%%%%%%
\subsection{Fourth-order method for undamped waves}
In the absence of damping, the wave equation corresponds to a separable Hamiltonian system. This fact explains the success of symplectic integrators,
such as the St\"ormer-Verlet or the leap-frog method, when combined with a symmetric discretization in space. Indeed the fully discrete numerical scheme
will then conserve (a discrete version of) the energy, too. Clearly, standard symplectic partitioned Runge-Kutta (Lobatto IIIA--IIIB pairs) or
composition methods \cite{GrMi_HLW02} can be used to achieve higher accuracy \cite{GrMi_RWR04}. Because the Hamiltonian here is separable,
those higher order versions will also remain explicit in time, like the St\"ormer-Verlet method. Since damped wave equations are linear,
we instead opt for the even more efficient modified equation (ME) approach \cite{GrMi_SB87} in this section, which leads to explicit LTS of arbitrarily high (even) order.

Following the ME approach, we replace ${\mathbf A}{\mathbf z}(t + \theta \, \Delta t)$ in (\ref{eq:GrMi_integral}) by its Taylor expansion
$$
{\mathbf A}{\mathbf z}(t + \theta \, \Delta t) = {\mathbf A} \left ({\mathbf z}(t) + \theta \, \Delta t \, {\mathbf z}'(t) +
 \frac{\theta^2 \, \Delta t^2}{2} \, {\mathbf z}''(t) +  \frac{\theta^3 \, \Delta t^3}{6} \, {\mathbf z}'''(t) \right ) +
{\mathcal O}(\Delta t^4) \,.
$$
Then, the integrals involving odd powers of $\theta$ vanish. Next, by using that ${\mathbf z}''(t) = {\mathbf R}(t) - {\mathbf A}{\mathbf z}(t)$ and
the Simpson quadrature rule for the term that involves ${\mathbf R}(t + \theta \, \Delta t)$, we obtain the fourth-order modified equation scheme.
\begin{equation} \begin{split} \label{eq:GrMi_mod_eq}
\frac{{\mathbf z}_{m+1} - 2{\mathbf z}_m + {\mathbf z}_{m-1} }{\Delta t^2} & =
{\mathbf R}_m - {\mathbf A}{\mathbf z}_m + \frac{\Delta t^2}{12} {\mathbf A}^2 {\mathbf z}_m - \frac{\Delta t^2}{12} {\mathbf A} {\mathbf R}_m  \\
& + \frac{1}{3} \left ( {\mathbf R}_{m-1/2} -2 {\mathbf R}_m + {\mathbf R}_{m+1/2} \right ) + {\mathcal O}(\Delta t^4) \,,
\end{split} \end{equation}
where ${\mathbf z}_m \simeq {\mathbf z}(t_m)$, ${\mathbf R}_m \simeq {\mathbf R}(t_m)$ and
${\mathbf R}_{m\pm1/2} \simeq {\mathbf R}(t_m \pm \Delta t/2)$. Clearly, integration schemes of arbitrary (even) order can be obtained by using
additional terms in the Taylor expansion. Since the maximal time-step allowed by the fourth-order ME method is about 70\% times {\it larger}
than that of the leap-frog scheme \cite{GrMi_CJRT01}, the additional work needed for the improved accuracy is quite small; hence, the ME method is extremely
efficient.

We now derive a fourth-order LTS method for (\ref{eq:GrMi_semi_disc_model_sigma_0}). Similarly to the derivation in Section \ref{sec:GrMi_LTS-LF2},
we split the vectors ${\mathbf z}(t)$ and ${\mathbf R}(t)$ in (\ref{eq:GrMi_integral2}) into a fine and a coarse part, and shall treat
${\mathbf z}^{[\mbox{\scriptsize fine}]}(t)$ and ${\mathbf R}^{[\mbox{\scriptsize fine}]}(t)$ differently from
${\mathbf z}^{[\mbox{\scriptsize coarse}]}(t)$ and ${\mathbf R}^{[\mbox{\scriptsize coarse}]}(t)$. We expand
${\mathbf z}^{[\mbox{\scriptsize coarse}]}(t + \theta \, \Delta t)$ in Taylor series as
\begin{equation*} \begin{split}
{\mathbf z}^{[\mbox{\scriptsize coarse}]}(t + \theta \, \Delta t) = {\mathbf z}^{[\mbox{\scriptsize coarse}]}(t)
& + \theta \, \Delta t \, \frac{d {\mathbf z}^{[\mbox{\scriptsize coarse}]}}{dt}(t)
+ \frac{\theta^2 \, \Delta t^2}{2} \, \frac{d^2 {\mathbf z}^{[\mbox{\scriptsize coarse}]}}{dt^2}(t) \\
& + \frac{\theta^3 \, \Delta t^3}{6} \, \frac{d^3 {\mathbf z}^{[\mbox{\scriptsize coarse}]}}{dt^3}(t) + {\mathcal O}(\Delta t^4)
\end{split} \end{equation*}
and insert it into (\ref{eq:GrMi_integral2}). In (\ref{eq:GrMi_integral2}), the integrals involving odd powers of $\theta$ vanish. By using
$$
\frac{d^2 {\mathbf z}^{[\mbox{\scriptsize coarse}]}}{dt^2}(t) = ({\mathbf I} - {\mathbf P}) \frac{d^2 {\mathbf z}}{dt^2}(t)
= ({\mathbf I} - {\mathbf P}) {\mathbf R}(t) - ({\mathbf I} - {\mathbf P}) {\mathbf A} {\mathbf z}(t)
$$
and the Simpson quadrature rule for the term in (\ref{eq:GrMi_integral2}) that involves
${\mathbf R}^{[\mbox{\scriptsize coarse}]}(t + \theta \Delta t)$, we find that
\begin{eqnarray} 
& & {\mathbf z}(t + \Delta t) - 2 \, {\mathbf z}(t) + {\mathbf z}(t - \Delta t) \nonumber \\
& & = \Delta t^2 \left \{ ({\mathbf I} - {\mathbf P}) {\mathbf R}(t) - {\mathbf A} ({\mathbf I} - {\mathbf P}) {\mathbf z}(t) \right \}
+ \frac{\Delta t^4}{12} {\mathbf A} ({\mathbf I} - {\mathbf P}) {\mathbf A} {\mathbf z}(t) \label{eq:GrMi_mod_int} \\
& & - \frac{\Delta t^4}{12} {\mathbf A} ({\mathbf I} - {\mathbf P}) {\mathbf R}(t)
+ \frac{\Delta t^2}{3} ({\mathbf I} - {\mathbf P})
\left \{ {\mathbf R}\left (t - \frac{\Delta t}{2} \right ) - 2 {\mathbf R}(t) + {\mathbf R}\left (t + \frac{\Delta t}{2} \right ) \right \} \nonumber \\
& & + \Delta t^2 \int_{-1}^1 (1 - \vert \theta \vert) \left \{{\mathbf R}^{[\mbox{\scriptsize fine}]}(t+\theta \Delta t)
- {\mathbf A} {\mathbf z}^{[\mbox{\scriptsize fine}]}(t+\theta \Delta t) \right \} \, d \theta \,. \nonumber 
\end{eqnarray}
Hence, if ${\mathbf P} = {\mathbf 0}$ we recover the standard ME scheme (\ref{eq:GrMi_mod_eq}).

Similarly to Section \ref{sec:GrMi_LTS-LF2}, we now approximate the right-hand side of (\ref{eq:GrMi_mod_int}) by solving the following differential equation for
${\widetilde {\mathbf z}}(\tau)$
\begin{equation*} \begin{split}
\frac{d^2 \widetilde{\mathbf z}}{d \tau^2}(\tau) & =
(\mathbf I - \mathbf P) {\mathbf R}(t) - {\mathbf A} (\mathbf I - \mathbf P) {\mathbf z}(t) \\
& + \frac{1}{3} ({\mathbf I} - {\mathbf P})
\left \{ {\mathbf R}\left (t - \frac{\Delta t}{2} \right ) - 2 {\mathbf R}(t) + {\mathbf R}\left (t + \frac{\Delta t}{2} \right ) \right \} \\
& + \frac{\tau^2}{2} {\mathbf A} ({\mathbf I} - {\mathbf P}) {\mathbf A} {\mathbf z}(t)
- \frac{\tau^2}{2} {\mathbf A} ({\mathbf I} - {\mathbf P}) {\mathbf R}(t)
+ {\mathbf P} {\mathbf R}(t + \tau) - {\mathbf A} {\mathbf P}  \widetilde{\mathbf z}(\tau) \,, \\
\widetilde{\mathbf z}(0) & = {\mathbf z}(t) \,, \ \widetilde{\mathbf z}'(0) = {\nu} \,,
\end{split} \end{equation*}
where $\nu$ will be specified below. Again, using Taylor expansions, we infer that
\begin{equation*}
{\mathbf z}(t + \Delta t) + {\mathbf z}(t - \Delta t) =
\widetilde{\mathbf z}(\Delta t) + \widetilde {\mathbf z}(- \Delta t) + {\mathcal O}(\Delta t^6) \,.
\end{equation*}
Again, the quantity $\widetilde{\mathbf z}(\Delta t) + \widetilde {\mathbf z}(- \Delta t)$ does not depend on the value of
$\nu$, which we set to zero. As in Section \ref{sec:GrMi_LTS-LF2}, we set $\mathbf q(\tau) := \widetilde{\mathbf z}(\tau) + \widetilde {\mathbf z}(- \tau)$, which solves
the differential equation
\begin{equation} \label{eq:GrMi_mod_diff_eq_q} \begin{split}
\frac{d^2 {\mathbf q}}{d \tau^2}(\tau) & = 2 \left \{
(\mathbf I - \mathbf P) {\mathbf R}(t) - {\mathbf A} (\mathbf I - \mathbf P) {\mathbf z}(t) \right \} \\
& + \frac{2}{3} ({\mathbf I} - {\mathbf P}) \left \{ {\mathbf R}\left (t - \frac{\Delta t}{2} \right ) - 2 {\mathbf R}(t) + {\mathbf R}\left (t + \frac{\Delta t}{2} \right ) \right \} \\
& + {\tau^2} {\mathbf A} ({\mathbf I} - {\mathbf P}) {\mathbf A} {\mathbf z}(t)
- {\tau^2} {\mathbf A} ({\mathbf I} - {\mathbf P}) {\mathbf R}(t) \\
& + {\mathbf P} \left \{ {\mathbf R}(t + \tau) + {\mathbf R}(t - \tau) \right \} \\
& - {\mathbf A} {\mathbf P} {\mathbf q}(\tau) \,, \\
{\mathbf q}(0) & = 2 {\mathbf z}(t) \,, \ {\mathbf q}'(0) = 0 \,.
\end{split} \end{equation}
Thus, we have
\begin{equation} \label{eq:GrMi_approx_eq_q_mod}
{\mathbf z}(t + \Delta t) + {\mathbf z}(t - \Delta t) = {\mathbf q}(\Delta t) + {\mathcal O}(\Delta t^6) \,.
\end{equation}
Now, we approximate the right side of (\ref{eq:GrMi_integral}) by solving (\ref{eq:GrMi_mod_diff_eq_q}) with the fourth-order ME method on $[0, \Delta t]$ with a
smaller time step $\Delta \tau = \Delta t / p$, and then use (\ref{eq:GrMi_approx_eq_q_mod}) to compute ${\mathbf z}(t + \Delta t)$.

In summary, the fourth-order LTS algorithm for the solution of (\ref{eq:GrMi_semi_disc_model_sigma_0}) computes
${\mathbf z}_{n+1}~\simeq~\mathbf z(t + \Delta t)$, given $\mathbf z_n$ and $\mathbf z_{n-1}$, as follows:

\bigskip

\noindent {\bf LTS-LFME4($p$) Algorithm}
{\it\begin{enumerate}
\item Set ${\mathbf q}_0 := 2 {\mathbf z}_n$,
$w_1 := ({\mathbf I} - {\mathbf P}){\mathbf R}_n - {\mathbf A}({\mathbf I} - {\mathbf P}) {\mathbf z}_n$,

$w_2 := {\mathbf A}({\mathbf I} - {\mathbf P}) {\mathbf A} {\mathbf z}_n - {\mathbf A}({\mathbf I} - {\mathbf P}) {\mathbf R}_n$ and
$r_1 := {\mathbf R}_{n-1/2} -2 {\mathbf R}_n + {\mathbf R}_{n+1/2}$.
\item Compute
\begin{equation*} \begin{split}
u & := 2w_1 + \frac{2}{3}({\mathbf I} - {\mathbf P})r_1 + 2 {\mathbf P}{\mathbf R}_{n, 0} - {\mathbf A}{\mathbf P}{\mathbf q}_0 \\
{\mathbf q}_{1/p} & := {\mathbf q}_0 + \frac{1}{2} \left ( \frac{\Delta t}{p} \right )^2 u
+ \frac{1}{24}  \left ( \frac{\Delta t}{p} \right )^4 \left ( 2w_2 + 2 \left ( \frac{2}{\Delta t} \right )^2 {\mathbf P} r_1 - {\mathbf A}{\mathbf P} u \right );
\end{split} \end{equation*}
\item For $m = 1, \dots, p-1$, compute
{\small \begin{equation*} \begin{split}
u_1 & := 2 w_1 + \frac{2}{3}({\mathbf I} - {\mathbf P})r_1 + \left ( \frac{m \Delta t}{p}\right )^2 w_2 +
{\mathbf P} \left ( {\mathbf R}_{n, m} + {\mathbf R}_{n, -m} \right ) - {\mathbf A}{\mathbf P}{\mathbf q}_{m/p}\,,\\
%%%
r & := {\mathbf R}_{n, m-1/2} - 2 {\mathbf R}_{n, m} + {\mathbf R}_{n, m+1/2} + {\mathbf R}_{n, -m-1/2} \\
  & \qquad \qquad \qquad - 2 {\mathbf R}_{n, -m} + {\mathbf R}_{n, -m+1/2} \,, \\
%%%
u_2 & := 2 w_2 + \left ( \frac{2p}{m\Delta t} \right )^2 {\mathbf P} r - {\mathbf A} {\mathbf P} u_1 \,,\\
%%%
{\mathbf q}_{(m+1)/p} & := 2 {\mathbf q}_{m/p} - {\mathbf q}_{(m-1)/p} + \left ( \frac{\Delta t}{p} \right )^2 u_1
+ \frac{1}{12} \left ( \frac{\Delta t}{p} \right )^4 u_2 \,.
\end{split} \end{equation*}}
\item Compute ${\mathbf z}_{n+1} := -{\mathbf z}_{n-1} + \mathbf q_1$.
\end{enumerate}
}
Here, Steps 1-3 correspond to the numerical solution of (\ref{eq:GrMi_mod_diff_eq_q}) until $\tau = \Delta t$ with the ME approach
using the local time-step $\Delta \tau = \Delta t /p$. The LTS-LFME4($p$) algorithm requires three -- two, without sources -- multiplications by
${\mathbf A}({\mathbf I} - {\mathbf P})$ and $2p$ further multiplications by ${\mathbf A}{\mathbf P}$.
For ${\mathbf P} = {\mathbf 0}$, that is without any local time-stepping, the algorithm reduces to the modified equation scheme (\ref{eq:GrMi_mod_eq}) above.

%%%%%%%%%%%%%%%%%%%%%%%%%%%%%%%%%%%%%%%%%%%%%%%%%%%%%%%%%%%%%%%%%%%%%%%%%%%%%%%%%%%%%%%%%%%%%%%%%%%%%%%%%%%%%%%%%%%%%%%%%
\subsection{Second-order leap-frog/Crank-Nicolson based method for damped waves}
We shall now derive a second-order LTS method for (\ref{eq:GrMi_semi_disc_model_final}) in a general form with $\mathbf D \not = \mathbf 0$.
In contrast to the time-stepping scheme presented in Section \ref{sec:GrMi_LTS-LF2} for the case $\mathbf D = \mathbf 0$, we are now faced with several difficulties
due to the additional ${\mathbf D} {\mathbf z}'(t)$ term. First, we shall treat that term implicitly to avoid any additional CFL restriction; else, the stability
condition will be more restrictive than that with the LTS-LF2($p$) scheme, depending on the magnitude of $\sigma$. Note that very large values of $\sigma$ will affect the CFL stability condition of any
explicit method regardless of the use of
local time-stepping. Nevertheless, the resulting scheme will be explicit,
since $\mathbf D$ is essentially a diagonal matrix. Second, we can no longer take advantage of any inherent symmetry in time of the solution.
Third, to avoid any loss of accuracy, we must carefully initialize the LTS scheme, which again is based on the highly efficient (two-step) leap-frog method.

The exact solution ${\mathbf z}(t)$ of (\ref{eq:GrMi_semi_disc_model_final}) satisfies
\begin{equation} \label{eq:GrMi_integral_approx} \begin{split}
& {\mathbf z}(t + \Delta t) - 2 \, {\mathbf z}(t)  + {\mathbf z}(t - \Delta t) + \frac{\Delta t}{2} {\mathbf D}
\left ({\mathbf z}(t + \Delta t) -  {\mathbf z}(t - \Delta t) \right ) \\
& = \Delta t^2 \, \int_{-1}^{1} (1 - \vert \theta \vert)
\left ( {\mathbf R}(t + \theta \, \Delta t) - {\mathbf A} \, {\mathbf z}(t + \theta \, \Delta t) \right ) \, d \theta 
+ {\mathcal O}(\Delta t^4) \,.
\end{split}\end{equation}
To derive  a second-order LTS method for (\ref{eq:GrMi_semi_disc_model_final}), we now split the vectors ${\mathbf z}(t)$ and ${\mathbf R}(t)$
as in (\ref{eq:GrMi_eq_split}) and approximate the integrands in (\ref{eq:GrMi_integral_approx}) as follows:
\begin{equation*} \begin{split}
{\mathbf R}^{[\mbox{\scriptsize coarse}]}(t + \theta \, \Delta t) + {\mathbf R}^{[\mbox{\scriptsize fine}]}
(t + \theta \, \Delta t) & \simeq {\mathbf R}^{[\mbox{\scriptsize coarse}]}(t) + {\mathbf P} {\mathbf R}(t + \theta \Delta t) \,, \\
{\mathbf A} \left ( {\mathbf z}^{[\mbox{\scriptsize coarse}]}(t + \theta \, \Delta t) + {\mathbf z}^{[\mbox{\scriptsize fine}]}(t + \theta \, \Delta t)
\right ) & \simeq {\mathbf A}{\mathbf z}^{[\mbox{\scriptsize coarse}]}(t) + {\mathbf A} {\mathbf P} {\mathbf z} (\theta \Delta t)\,.
\end{split} \end{equation*}
We thus have
\begin{equation} \label{eq:GrMi_integral3_sigma} \begin{split}
{\mathbf z}(t + \Delta t) & - 2 \, {\mathbf z}(t) + {\mathbf z}(t - \Delta t) + \frac{\Delta t}{2} {\mathbf D} 
\left ({\mathbf z}(t + \Delta t) -  {\mathbf z}(t - \Delta t) \right ) \\
& \simeq \Delta t^2 \left \{
({\mathbf I} - {\mathbf P}) {\mathbf R}(t) - {\mathbf A} ({\mathbf I} - {\mathbf P}) {\mathbf z}(t)\right \} \\
& + \Delta t^2 \, \int_{-1}^{1} (1 - \vert \theta \vert) 
\left \{ {\mathbf P} {\mathbf R}(t + \theta \Delta t) - {\mathbf A} {\mathbf P} {\mathbf z}(\theta \Delta t) \right \}
 \, d \theta \,.
\end{split} \end{equation}
Next for fixed $t$, let ${\widetilde {\mathbf z}}(\tau)$ solve the differential equation
\begin{equation} \label{eq:GrMi_diff_eq_z_tilde} \begin{split}
\frac{d^2 \widetilde{\mathbf z}}{d \tau^2}(\tau) + {\mathbf D} \frac{d \widetilde{\mathbf z}}{d \tau}(\tau)& =
(\mathbf I - \mathbf P) {\mathbf R}(t) - {\mathbf A} (\mathbf I - \mathbf P) {\mathbf z}(t) +
{\mathbf P} {\mathbf R}(t + \tau) \\ & - {\mathbf A} {\mathbf P}  \widetilde{\mathbf z}(\tau) \,, \\
\widetilde{\mathbf z}(0) & = {\mathbf z}(t) \,, \ \widetilde{\mathbf z}'(0) = {\nu} \,,
\end{split} \end{equation}
where $\nu$ will be specified below. Since the exact solution $\widetilde{\mathbf z}(t)$ of (\ref{eq:GrMi_diff_eq_z_tilde}) satisfies
\begin{equation} \label{eq:GrMi_integral4_sigma} \begin{split}
\widetilde {\mathbf z}(\Delta t) & - 2 \, \widetilde {\mathbf z}(0) + \widetilde {\mathbf z}(- \Delta t) +
\frac{\Delta t}{2} {\mathbf D} \left (  \widetilde {\mathbf z}(\Delta t) - \widetilde {\mathbf z}(- \Delta t) \right ) \\
& = \Delta t^2 \left \{ ({\mathbf I} - {\mathbf P}) {\mathbf R}(t) - {\mathbf A} ({\mathbf I} - {\mathbf P}) {\mathbf z}(t)\right \} \\
& + \Delta t^2 \, \int_{-1}^{1} (1 - \vert \theta \vert)
\left \{ {\mathbf P} {\mathbf R}(t + \theta \Delta t) - {\mathbf A} {\mathbf P} \widetilde{\mathbf z}(\theta \Delta t) \right \}
 \, d \theta \,,
\end{split} \end{equation}
from the comparison of (\ref{eq:GrMi_integral3_sigma}) and (\ref{eq:GrMi_integral4_sigma}), we infer that
\begin{equation} \label{eq:GrMi_approx_eq_sigma_1} \begin{split}
{\mathbf z}(t + \Delta t) & + {\mathbf z}(t - \Delta t) + \frac{\Delta t}{2}
{\mathbf D} \left ({\mathbf z}(t + \Delta t) -  {\mathbf z}(t - \Delta t) \right ) \\
& \simeq  \widetilde {\mathbf z}(\Delta t) + \widetilde {\mathbf z}(- \Delta t) + \frac{\Delta t}{2} {\mathbf D} \left (  \widetilde {\mathbf z}(\Delta t) -
\widetilde {\mathbf z}(- \Delta t) \right )\,.
\end{split} \end{equation}

In our local time-stepping scheme, we shall use the right side of (\ref{eq:GrMi_approx_eq_sigma_1}) to approximate the left side. In doing so, we must
carefully choose $\nu$ in (\ref{eq:GrMi_diff_eq_z_tilde}) to minimize that approximation error. By using Taylor expansions and the fact that $\mathbf z$ and
$\widetilde{\mathbf z}$ solve (\ref{eq:GrMi_semi_disc_model_final}) and (\ref{eq:GrMi_diff_eq_z_tilde}), respectively, we obtain
\begin{equation*} \begin{split}
{\mathbf z}(t + \Delta t) + {\mathbf z}(t - \Delta t) & = 2 {\mathbf z}(t) + {\mathbf z}''(t) \Delta t^2 + {\mathcal O}(\Delta t^4) \\
& = 2 {\mathbf z}(t) + ({\mathbf R}(t) - {\mathbf A} {\mathbf z}(t) - {\mathbf D} \, {\mathbf z}'(t)) \Delta t^2 + {\mathcal O}(\Delta t^4)\,, \\
\widetilde{\mathbf z}(\Delta t) + \widetilde {\mathbf z}(- \Delta t) & = 2 \widetilde {\mathbf z}(0) +
\widetilde {\mathbf z}''(0) \Delta t^2 + {\mathcal O}(\Delta t^4) \\
& = 2 {\mathbf z}(t) + ({\mathbf R}(t) - {\mathbf A} {\mathbf z}(t) - {\mathbf D} \, \nu) \Delta t^2 + {\mathcal O}(\Delta t^4) \,,\\
\end{split} \end{equation*}
together with
\begin{equation*} \begin{split}
{\mathbf z}(t+\Delta t) - {\mathbf z}(t-\Delta t) = 2  {\mathbf z}'(t) \, \Delta t + {\mathcal O}(\Delta t^3)\,, \
\widetilde{\mathbf z}(\Delta t) - \widetilde {\mathbf z}(- \Delta t) = 2  \nu \, \Delta t + {\mathcal O}(\Delta t^3)\,.
\end{split} \end{equation*}
Hence for arbitrary $\nu$, the right side of (\ref{eq:GrMi_approx_eq_sigma_1}) is not sufficiently accurate to approximate the left side while preserving
overall second-order accuracy. However, if we choose
$$
\nu = {\mathbf z}'(t)
$$
in (\ref{eq:GrMi_diff_eq_z_tilde}), the ${\mathcal O}(\Delta t^2)$ terms in (\ref{eq:GrMi_approx_eq_sigma_1}) cancel each other and overall second-order accuracy of the
scheme can be achieved. Since the term on the right side of (\ref{eq:GrMi_approx_eq_sigma_1}) is not symmetric in time, unlike in the previous
section (see (\ref{eq:GrMi_approx_eq}) and (\ref{eq:GrMi_approx_eq_q})), we need to compute the value of $\widetilde {\mathbf z}(\tau)$ both at $\tau = \Delta t$ and
at $\tau = -\Delta t$.

For the numerical solution of (\ref{eq:GrMi_diff_eq_z_tilde}), we shall use the leap-frog scheme with the local time-step $\Delta \tau = \Delta t /p$. Since
the leap-frog scheme is a two-step method, we need a second-order approximation of $\widetilde{\mathbf z}'(0) = {\mathbf z}'(t)$ during every initial
local time-step. Since the value of ${\mathbf z}_{n+1}$ is still unknown at time $t = t_n$, we now derive a second-order approximation
${\mathbf z}'_n \simeq {\mathbf z}'(t)$ that uses only ${\mathbf z}_{n}$ and ${\mathbf z}_{n-1}$. First, we approximate
\begin{equation} \label{eq:GrMi_approz_d_z}
{\mathbf z}'_n \simeq \frac{{\mathbf z}'_{n-1/2} + {\mathbf z}'_{n+1/2}}{2} \,,
\end{equation}
where both ${\mathbf z}'_{n-1/2} \simeq {\mathbf z}'(t - \Delta t/2)$ and ${\mathbf z}'_{n+1/2} \simeq {\mathbf z}'(t + \Delta t /2)$ are second-order
approximations. By using second-order central differences for ${\mathbf z}'_{n-1/2}$,
\begin{equation} \label{eq:GrMi_approx_d_z_minus}
{\mathbf z}'_{n-1/2} = \frac{{\mathbf z}_n - {\mathbf z}_{n-1}}{\Delta t} + {\mathcal O}(\Delta t^2) \,,
\end{equation}
and the differential equation (\ref{eq:GrMi_semi_disc_model_final}) for ${\mathbf z}'_{n+1/2}$,
$$
\frac{{\mathbf z}'_{n+1/2} - {\mathbf z}'_{n-1/2}}{\Delta t} + {\mathbf D} {\mathbf z}'_n =
{\mathbf R}_n - {\mathbf A} {\mathbf z}_n + {\mathcal O}(\Delta t^2) \,,
$$
we obtain
\begin{equation} \label{eq:GrMi_approx_d_z_plus}
{\mathbf z}'_{n+1/2} = \left ( {\mathbf I} + \frac{\Delta t}{2} {\mathbf D} \right )^{-1} \left \{
\left ( {\mathbf I} - \frac{\Delta t}{2} {\mathbf D} \right ) \frac{{\mathbf z}_n - {\mathbf z}_{n-1}}{\Delta t} +
\Delta t {\mathbf R}_n - \Delta t {\mathbf A} {\mathbf z}_n \right \} + {\mathcal O}(\Delta t^2) \,.
\end{equation}
Then, we insert (\ref{eq:GrMi_approx_d_z_minus}), (\ref{eq:GrMi_approx_d_z_plus}) into (\ref{eq:GrMi_approz_d_z}),
which yields a second-order approximation of ${\mathbf z}'(t)$.

In summary, the second-order LTS algorithm for the solution of (\ref{eq:GrMi_semi_disc_model_final}) computes
${\mathbf z}_{n+1} \simeq \mathbf z(t+\Delta t)$, for given $\mathbf z_n$ and $\mathbf z_{n-1}$, as follows:

\bigskip

\noindent {\bf LTS-LFCN2($p$) Algorithm}
{\it\begin{enumerate}
\item Set ${\mathbf w} := ({\mathbf I} - {\mathbf P}){\mathbf R}_n - {\mathbf A}({\mathbf I} - {\mathbf P}) {\mathbf z}_n$,
${\widetilde{\mathbf z}}_0 := {\mathbf z}_n$ and
\begin{equation*} \begin{split}
{\mathbf z}'_n := \frac{1}{2} & \left [ \frac{{\mathbf z}_n - {\mathbf z}_{n-1}}{\Delta t} \right. \\
& \left. + \left ( {\mathbf I} + \frac{\Delta t}{2} {\mathbf D} \right )^{-1}
\left \{ \left ( {\mathbf I} - \frac{\Delta t}{2} {\mathbf D} \right ) \frac{{\mathbf z}_n - {\mathbf z}_{n-1}}{\Delta t} +
\Delta t {\mathbf R}_n - \Delta t {\mathbf A} {\mathbf z}_n \right \} \right ] \,.
\end{split} \end{equation*}
%%%
\item Compute
\begin{equation*} \begin{split}
{\widetilde {\mathbf z}}_{1/p} & := {\widetilde{\mathbf z}}_0 + \frac{\Delta t}{p} {\mathbf z}'_n +
\frac{1}{2} \left ( \frac{\Delta t}{p} \right )^2
\left ( \mathbf w + {\mathbf P} {\mathbf R}_{n, 0} - {\mathbf A} {\mathbf P} {\widetilde {\mathbf z}}_0 
- {\mathbf D} {\mathbf z}'_n \right ) \quad \mbox{and} \\
{\widetilde {\mathbf z}}_{-1/p} & := {\widetilde{\mathbf z}}_0 - \frac{\Delta t}{p} {\mathbf z}'_n +
\frac{1}{2} \left ( \frac{\Delta t}{p} \right )^2
\left ( \mathbf w + {\mathbf P} {\mathbf R}_{n, 0} - {\mathbf A} {\mathbf P} {\widetilde {\mathbf z}}_0 
- {\mathbf D} {\mathbf z}'_n \right ) \,.
\end{split} \end{equation*}

\item For $m = 1, \dots, p-1$, compute
\begin{equation*} \begin{split}
{\widetilde {\mathbf z}}_{(m+1)/p} := \left ( {\mathbf I} + \frac{\Delta t}{2p} {\mathbf D} \right )^{-1} & \left \{
2 {\widetilde {\mathbf z}}_{m/p} -
\left ( {\mathbf I} - \frac{\Delta t}{2p} {\mathbf D} \right ) {\widetilde {\mathbf z}}_{(m-1)/p} \right. \\
& \left. + \left ( \frac{\Delta t}{p} \right )^2 ( \mathbf w + {\mathbf P}{\mathbf R}_{n, m} -
{\mathbf A} {\mathbf P} {\widetilde {\mathbf z}}_{m/p} ) \right \}
\end{split} \end{equation*}
and
\begin{equation*} \begin{split}
{\widetilde {\mathbf z}}_{-(m+1)/p} := \left ( {\mathbf I} - \frac{\Delta t}{2p} {\mathbf D} \right )^{-1} & \left \{
2 {\widetilde {\mathbf z}}_{-m/p} -
\left ( {\mathbf I} + \frac{\Delta t}{2p} {\mathbf D} \right ) {\widetilde {\mathbf z}}_{-(m-1)/p} \right. \\
& \left. + \left ( \frac{\Delta t}{p} \right )^2 ( \mathbf w + {\mathbf P}{\mathbf R}_{n, -m} -
{\mathbf A} {\mathbf P} {\widetilde {\mathbf z}}_{-m/p} ) \right \} \,.
\end{split} \end{equation*}

\item Compute
$${\mathbf z}_{n+1} := \widetilde{\mathbf z}_{1} +  \left ( {\mathbf I} + \frac{\Delta t}{2} {\mathbf D} \right )^{-1}
\left ( {\mathbf I} - \frac{\Delta t}{2} {\mathbf D} \right ) \left ( -{\mathbf z}_{n-1} + \widetilde{\mathbf z}_{-1} \right ) \,.$$
\end{enumerate}}

If $\sigma$ is piecewise constant in each element, ${\mathbf M}$ and ${\mathbf M}_{\sigma}$ can be diagonalized simultaneously
and hence the matrix ${\mathbf D}$ is diagonal.
If $\sigma$ varies in elements, $\mathbf D$ is a block-diagonal matrix and both $\left ( {\mathbf I} \pm ({\Delta t}/{2p}) {\mathbf D} \right)$
and $\left ( {\mathbf I} \pm ({\Delta t}/{2}) {\mathbf D} \right)$ can be explicitly inverted at low cost. In that sense, the LTS-LFCN2($p$) scheme 
is truly explicit. Again, if the fraction of nonzero entries in $\mathbf P$ is small, the overall cost is dominated by the computation of 
$\mathbf w$ in step~1.

\begin{proposition}
For ${\mathbf R}(t) \in C^2([0, T])$, the local time-stepping method LTS-LFCN2 is second-order accurate.
\end{proposition}
\begin{proof} See \cite{GrMi_GM10}.\end{proof}

\begin{remark}
For $\sigma = 0$ (${\mathbf D} = {\mathbf 0}$), the LTS-LFCN2($p$) algorithm  coincides with the LTS-LF2($p$) algorithm and thus also conserves the discrete energy
(\ref{eq:GrMi_disc_energy_LT}). For $\sigma \not = 0$ and $p=1$, i.e. no local mesh refinement, one can easily show that the energy is no longer
conserved but decays with time (independently of $\sigma$) under the same CFL condition as in the case with $\sigma = 0$.
\end{remark}
%%%%%%%%%%%%%%%%%%%%%%%%%%%%%%%%%%%%%%%%%%%%%%%%%%%%%%%%%%%%%%%%%%%%%%%%%%%%%%%%%%%%%%%%%%%%%%%%%%%%%%%%%%%%%%%%%%%%%%%%%

%%%%%%%%% SECTION 4 %%%%%%%%%%%%%%%%%%%%%%%%%%%%%%%%%%%%%%%%%%%%%%%%%%%%%%%%%%%%%%%%%%%%%%%%%%%%%%%%%%%%%%%%%%%%%%%%%%%%%
\section{Adams-Bashforth based LTS methods for damped waves}
Starting from the standard leap-frog method, we proposed in Section \ref{sec:GrMi_LTS-LF} energy conserving fully explicit LTS integrators of arbitrarily high accuracy
for undamped waves. By blending the leap-frog and the Crank-Nicolson methods, a second-order LTS scheme was also derived there for damped waves, yet this approach
cannot be readily extended beyond order two. To achieve arbitrarily high accuracy in the presence of damping, while remaining fully explicit, we shall derive here
explicit LTS methods for damped wave equations based on Adams-Bashforth (AB) multi-step schemes.

The $H^1$-conforming and the IP-DG finite element discretizations of (\ref{eq:GrMi_model_eq_1}) presented in Section \ref{sec:GrMi_fem_wave} lead to the
second-order system of differential equations (\ref{eq:GrMi_sdisc_cfem}), whereas the nodal DG discretization leads to the first-order system of differential equations
(\ref{eq:GrMi_sd_pr2}). In both (\ref{eq:GrMi_sdisc_cfem}) and (\ref{eq:GrMi_sd_pr2}) the mass matrix ${\mathbf M}$ is symmetric, positive definite and
essentially diagonal; thus, ${\mathbf M}^{-1}$ or ${\mathbf M}^{-\frac{1}{2}}$ can be computed explicitly at a negligible cost. For simplicity, we restrict 
ourselves here to the homogeneous case, i.e. ${\mathbf F}(t) = \mathbf{0}$.

If we multiply (\ref{eq:GrMi_sdisc_cfem}) by ${\mathbf M}^{-\frac{1}{2}}$, we obtain (\ref{eq:GrMi_semi_disc_model_final}).
Thus, we can rewrite (\ref{eq:GrMi_semi_disc_model_final}) as a first-order problem of the form
\begin{equation} \label{eq:GrMi_sd_mod_pr}
\frac{d {\mathbf y}}{d t}(t) = {\mathbf B} {\mathbf y}(t)\,,
\end{equation} with
$$
{\mathbf y}(t) = \left ( {\mathbf z}(t), \frac{d {\mathbf z}}{d t}(t) \right )^T\,, \qquad {\mathbf B} = \left ( \begin{array}{rr}
{\mathbf 0} & {\mathbf I} \\ -{\mathbf A} & -{\mathbf D} \end{array} \right ) \,.
$$
Similarly, we can also rewrite (\ref{eq:GrMi_sd_pr2}) in the form (\ref{eq:GrMi_sd_mod_pr}) with
${\mathbf y}(t) = {\mathbf Q}(t)$ and ${\mathbf B} = {\mathbf M}^{-1}\left ( -{\mathbf M}_\sigma - {\mathbf C} \right )$. Hence all three distinct finite element
discretizations from Section \ref{sec:GrMi_fem_wave} lead to a semi-discrete system as in (\ref{eq:GrMi_sd_mod_pr}). Starting from explicit multi-step AB methods,
we shall now derive explicit LTS schemes of arbitrarily high accuracy for a general problem of the form (\ref{eq:GrMi_sd_mod_pr}).

First, we briefly recall the construction of the classical $k$-step ($k$th-order) Adams-Bashforth method for the numerical solution of
(\ref{eq:GrMi_sd_mod_pr}) \cite{GrMi_HNW00}. Let $t_i = i \Delta t$ and ${\mathbf y}_n$, ${\mathbf y}_{n-1}$,..., ${\mathbf y}_{n-k+1}$ the numerical approximations
to the exact solution ${\mathbf y}(t_n)$, $\dots$, \linebreak ${\mathbf y}(t_{n-k+1})$. The solution of (\ref{eq:GrMi_sd_mod_pr}) satisfies
\begin{equation} \label{eq:GrMi_int_eq}
{\mathbf y}(t_n + \xi \Delta t) = {\mathbf y}(t_n) + \int_{t_n}^{t_n + \xi \Delta t} {\mathbf B} {\mathbf y}(t) \, dt \,, \qquad 0 < \xi \leq 1\,.
\end{equation}
We now replace the unknown solution ${\mathbf y}(t)$ under the integral in (\ref{eq:GrMi_int_eq}) by the interpolation polynomial $p(t)$ through the points
$(t_i, {\mathbf y}_i)$, $i = n-k+1, \dots, n$. It is explicitly given in terms of backward differences
$$
\nabla^0 {\mathbf y}_n = {\mathbf y}_n\,, \quad \nabla^{j+1} {\mathbf y}_n = \nabla^{j} {\mathbf y}_n - \nabla^{j} {\mathbf y}_{n-1}
$$
by
\begin{equation*}
p(t) = p(t_n + s \Delta t) = \sum_{j = 0}^{k-1}(-1)^j \left (\begin{array}{c} -s \\ j \end{array} \right ) \nabla^{j} {\mathbf y}_n \,.
\end{equation*}
Integration of (\ref{eq:GrMi_int_eq}) with ${\mathbf y}(t)$ replaced by $p(t)$ then yields the approximation ${\mathbf y}_{n+\xi}$ of
${\mathbf y}(t_n + \xi \Delta t)$, $0 < \xi \leq 1$,
\begin{equation} \label{eq:GrMi_ab}
{\mathbf y}_{n+\xi} = {\mathbf y}_n + \Delta t {\mathbf B} \sum_{j=0}^{k-1} \gamma_j (\xi) \nabla^{j} {\mathbf y}_n\,,
\end{equation}
where the polynomials $\gamma_j(\xi)$ are defined as
$$
\gamma_j(\xi) = (-1)^j \int_0^{\xi} \left ( \begin{array}{c} -s \\ j \end{array} \right ) \, ds \,.
$$
They are given in Table \ref{tab:GrMi_tab1} for $j \leq 3$. After expressing the backward differences in terms of ${\mathbf y}_{n-j}$ and setting $\xi = 1$ in
(\ref{eq:GrMi_ab}), we recover the common form of the $k$-step Adams-Bashforth scheme \cite{GrMi_HNW00}
\begin{equation} \label{eq:GrMi_stand_ab}
{\mathbf y}_{n+1} = {\mathbf y}_n + \Delta t {\mathbf B} \sum_{j=0}^{k-1} \alpha_j {\mathbf y}_{n-j} \,,
\end{equation}
where the coefficients $\alpha_j$, $j = 0,\dots, k-1$ for the second, third- and fourth-order ($k=2,3,4$) Adams-Bashforth schemes are given in Table
\ref{tab:GrMi_tab2}. For higher values of $k$ we refer to \cite{GrMi_HNW00}.

\begin{table}[h!]
\begin{center}
\renewcommand{\arraystretch}{1.5}
\begin{tabular}{c|cccc} \hline
$j$ & 0 & 1 & 2 & 3  \\ \hline $\gamma_j(\xi)$ & $\xi$ &
$\frac{1}{2} \xi^2$ & $\frac{1}{6} \xi^3 + \frac{1}{4} \xi^2$ &
$\frac{1}{24} \xi^4 + \frac{1}{6} \xi^3 + \frac{1}{6} \xi^2$ \\
\hline
\end{tabular}
\caption{Coefficients $\gamma_j(\xi)$ for the explicit Adams-Bashforth methods.} \label{tab:GrMi_tab1}
\end{center}
\end{table}
\begin{table}[ht!]
\begin{center}
\renewcommand{\arraystretch}{1.5}
\begin{tabular}{c|cccc} \hline
      & $\alpha_0$      & $\alpha_1$       & $\alpha_2$      & $\alpha_3$      \\ \hline
$k=2$ & $\frac{3}{2}$   & $-\frac{1}{2}$   & 0               & 0
\\ \hline $k=3$ & $\frac{23}{12}$ & $-\frac{16}{12}$ &
$\frac{5}{12}$  & 0               \\ \hline $k=4$ & $\frac{55}{24}$
& $-\frac{59}{24}$ & $\frac{37}{24}$ & $-\frac{9}{24}$ \\ \hline
\end{tabular}
\caption{Coefficients for the $k$-th order Adams-Bashforth methods.} \label{tab:GrMi_tab2}
\end{center}
\end{table}

Starting from the classical AB methods, we shall now derive LTS schemes of arbitrarily high accuracy for (\ref{eq:GrMi_sd_mod_pr}), which allow arbitrarily
small time-steps precisely where small elements in the spatial mesh are located. To do so, we first split the unknown vector ${\mathbf y}(t)$ in two parts
\begin{equation*}
{\mathbf y}(t) = ({\mathbf I} - {\mathbf P}) {\mathbf y}(t) + {\mathbf P} {\mathbf y}(t) =
{\mathbf y}^{[\mbox{\scriptsize{coarse}}]}(t) + {\mathbf y}^{[\mbox{\scriptsize fine}]}(t)\,,
\end{equation*}
where the matrix $\mathbf P$ is diagonal. Its diagonal entries, equal to zero or one, identify the unknowns associated with the locally refined region, where smaller
time-steps are needed.

The exact solution of (\ref{eq:GrMi_sd_mod_pr}) again satisfies
\begin{equation} \label{eq:GrMi_tmp1}
{\mathbf y}(t_n + \xi \Delta t) = {\mathbf y}(t_n) + \int_{t_n}^{t_n + \xi \Delta t} {\mathbf B}
\left ( {\mathbf y}^{[\mbox{\scriptsize{coarse}}]}(t) + {\mathbf y}^{[\mbox{\scriptsize fine}]}(t) \right ) \,
dt \,, \qquad 0 < \xi \leq 1\,.
\end{equation}
Since we wish to use the standard $k$-step Adams-Bashforth method in the coarse region, we approximate the term in (\ref{eq:GrMi_tmp1}) that involve ${\mathbf
y}^{[\mbox{\scriptsize coarse}]}(t)$ as in (\ref{eq:GrMi_int_eq}), which yields
\begin{equation} \label{eq:GrMi_tmp2}
{\mathbf y}(t_n + \xi \Delta t) \approx {\mathbf y}_n + \Delta t \, {\mathbf B}
({\mathbf I} - {\mathbf P}) \sum_{j=0}^{k-1} \gamma_j (\xi) \nabla^{j} {\mathbf y}_n + \int_{t_n}^{t_n + \xi \Delta t}
{\mathbf B}{\mathbf P} {\mathbf y}(t) \, dt \,.
\end{equation}
To circumvent the severe stability constraint due to the smallest elements associated with ${\mathbf y}^{[\mbox{\scriptsize fine}]}(t)$, we shall now treat
${\mathbf y}^{[\mbox{\scriptsize fine}]}(t)$ differently from ${\mathbf y}^{[\mbox{\scriptsize coarse}]}(t)$. Hence, we instead approximate the integrand in
(\ref{eq:GrMi_tmp2}) as
\begin{equation*}
\int_{t_n}^{t_n + \xi \Delta t} {\mathbf B}{\mathbf P} {\mathbf y}(t) \, dt \approx
\int_{0}^{\xi \Delta t} {\mathbf B}{\mathbf P} {\widetilde{\mathbf y}}(\tau) \, d \tau \,,
\end{equation*}
where ${\widetilde {\mathbf y}}(\tau)$ solves the differential
equation
\begin{equation} \label{eq:GrMi_mod_pr} \begin{split}
\frac{d {\widetilde {\mathbf y}}}{d \tau}(\tau) & =
{\mathbf B}({\mathbf I} - {\mathbf P}) \sum_{j=0}^{k-1} {\widetilde \gamma}_j
\left ( \frac{\tau}{\Delta t} \right ) \nabla^{j} {\mathbf y}_n +
{\mathbf B}{\mathbf P} \, {\widetilde {\mathbf y}}(\tau) \,,\\
{\widetilde {\mathbf y}}(0) & = {\mathbf y}_n\,,
\end{split} \end{equation}
with coefficients
\begin{equation} \label{eq:GrMi_coeff_gamma_lts}
{\widetilde \gamma}_j (\xi) = \frac{d}{d \xi} \gamma_j ({\xi}) = \frac{d}{d \xi} \left ( (-1)^j \int_0^\xi \binom{-s}{j} \, ds \right ) =
(-1)^j \binom{-\xi}{j}\,.
\end{equation}
The polynomials ${\widetilde \gamma}_j (\xi)$ are given in Table \ref{tab:GrMi_tab3} for $j \leq 3$.
\begin{table}[t!]
\begin{center}
\renewcommand{\arraystretch}{1.5}
\begin{tabular}{c|cccc} \hline
$j$ & 0 & 1 & 2 & 3  \\ \hline ${\widetilde \gamma}_j(\xi)$ & $1$ &
$\xi$ & $\frac{1}{2} \xi^2 + \frac{1}{2} \xi$ & $\frac{1}{6} \xi^3 +
\frac{1}{2} \xi^2 + \frac{1}{3} \xi$ \\ \hline
\end{tabular}
\caption{The polynomial coefficients ${\widetilde \gamma_j}(\xi)$} \label{tab:GrMi_tab3}
\end{center}
\end{table}
Replacing ${\mathbf y}(t)$ by ${\widetilde {\mathbf y}}(t)$ in (\ref{eq:GrMi_tmp2}), we obtain
\begin{equation} \label{eq:GrMi_tmp4}
{\mathbf y}(t_n + \xi \Delta t) \approx {\mathbf y}_n + \Delta t \, {\mathbf B}
({\mathbf I} - {\mathbf P}) \sum_{j=0}^{k-1} \gamma_j (\xi) \nabla^{j} {\mathbf y}_n +
\int_{0}^{\xi \Delta t} {\mathbf B}{\mathbf P} {\widetilde{\mathbf y}}(\tau) \, d \tau \,.
\end{equation}
By considering (\ref{eq:GrMi_mod_pr}) in integrated form, we find that
\begin{equation} \label{eq:GrMi_tmp3} \begin{split}
{\widetilde {\mathbf y}}(\xi \Delta t) & =
{\widetilde {\mathbf y}}(0) + {\mathbf B}({\mathbf I} - {\mathbf P})
\sum_{j=0}^{k-1} \left (
\int_0^{\xi \Delta t} {\widetilde \gamma}_j \left( \frac{\tau}{\Delta t} \right) \, d \tau \right )
\nabla^{j} {\mathbf y}_n
+ \int_{0}^{\xi \Delta t} {\mathbf B}{\mathbf P} \, {\widetilde {\mathbf y}}(\tau) \, d \tau \\
& = {\mathbf y}_n + \Delta t \, {\mathbf B}({\mathbf I} - {\mathbf P})
\sum_{j=0}^{k-1}  \gamma_j (\xi) \nabla^{j} {\mathbf y}_n +
\int_{0}^{\xi \Delta t} {\mathbf B}{\mathbf P} \, {\widetilde {\mathbf y}}(\tau) \, d \tau \,.
\end{split} \end{equation}
From the comparison of (\ref{eq:GrMi_tmp4}) and (\ref{eq:GrMi_tmp3}) we infer that
$${\mathbf y}(t_n + \xi \Delta t) \approx {\widetilde {\mathbf y}}(\xi \Delta t) \,.$$
Thus to advance ${\mathbf y}(t_n)$ from $t_n$ to $t_n + \Delta t$, we shall evaluate $\widetilde{\mathbf y}(\Delta t)$ by
solving (\ref{eq:GrMi_mod_pr}) on $[0, \Delta t]$ numerically. We solve (\ref{eq:GrMi_mod_pr}) until $\tau = \Delta t$ again with a $k$-step
Adams-Bashforth scheme, using a smaller time-step $\Delta \tau = \Delta t / p$, where $p$ denotes the ratio of local
refinement. For $m = 0, \dots, p-1$ we then have
\begin{equation} \begin{split} \label{eq:GrMi_tmp5}
\widetilde{\mathbf y}_{(m+1)/p} = \widetilde{\mathbf y}_{m/p} & +
\Delta \tau \, {\mathbf B}({\mathbf I} - {\mathbf P})
\sum_{\ell = 0}^{k-1} \alpha_{\ell} \sum_{j = 0}^{k-1} {\widetilde \gamma}_j \left ( \frac{m-\ell}{p} \right )
\nabla^{j} {\mathbf y}_n \\
& + \Delta \tau \, {\mathbf B}{\mathbf P} \sum_{\ell = 0}^{k-1} \alpha_{\ell} {\widetilde{\mathbf y}}_{(m-l)/p} \,,
\end{split} \end{equation}
where $\alpha_\ell$, $\ell = 0,\dots, k-1$ denote the coefficients of the classical $k$-step AB scheme
(see Table \ref{tab:GrMi_tab2}). Finally, after expressing the backward differences in terms of
${\mathbf y}_{n-\ell}$, we find
\begin{equation} \label{eq:GrMi_lts_ab}
\widetilde{\mathbf y}_{(m+1)/p} = \widetilde{\mathbf y}_{m/p} +
\Delta \tau \, {\mathbf B}({\mathbf I} - {\mathbf P}) \sum_{\ell = 0}^{k-1} \beta_{m, \ell} \, {\mathbf y}_{n-\ell} +
\Delta \tau \, {\mathbf B}{\mathbf P} \sum_{\ell = 0}^{k-1} \alpha_{\ell} \, {\widetilde{\mathbf y}}_{(m-l)/p} \,,
\end{equation}
where the constant coefficients $\beta_{m, \ell}$, $m=0,\dots,p-1$, $\ell= 0,\dots,k-1$, satisfy
\begin{equation} \label{eq:GrMi_coeff_beta_lts}
\beta_{m, \ell} = \sum_{i=0}^{k-1} \alpha_i
\sum_{j = \ell}^{k-1} (-1)^{\ell} \left ( \begin{array}{c} j \\ \ell \end{array} \right )
{\widetilde \gamma}_j \left ( \frac{m-i}{p} \right ) \,,
\end{equation}
with ${\widetilde \gamma}_j$ defined in (\ref{eq:GrMi_coeff_gamma_lts}).

In summary, the LTS-AB$k$($p$) algorithm computes ${\mathbf y}_{n+1} \simeq \mathbf y(t_n +
\Delta t)$, given $\mathbf y_n$, ${\mathbf y}_{n-1}$,..., ${\mathbf y}_{n-k+1}$,
${\mathbf B}({\mathbf I} - {\mathbf P}) {\mathbf y}_{n-1}, \dots$, ${\mathbf B}({\mathbf I} - {\mathbf P}) {\mathbf y}_{n-k+1}$
and ${\mathbf P}{\mathbf y}_{n-1/p}$, ${\mathbf P}{\mathbf y}_{n-2/p}$, $\dots$, ${\mathbf P}{\mathbf y}_{n-(k-1)/p}$ as follows:

\bigskip

\noindent {\bf LTS-AB$k$($p$) Algorithm}{\it
\begin{enumerate}
\item Set ${\widetilde{\mathbf y}}_0:= \mathbf y_n$, ${\widetilde{\mathbf y}}_{-\ell/p}:= {\mathbf P}{\mathbf y}_{n-\ell/p}$,
$\ell = 1, \dots, k-1$.
\item Set ${\mathbf w}_{n-\ell} := {\mathbf B}({\mathbf I} - {\mathbf P}) {\mathbf y}_{n-\ell}$, $\ell = 1, \dots, k-1$.
\item Compute ${\mathbf w}_{n} := {\mathbf B}({\mathbf I} - {\mathbf P}) {\mathbf y}_{n}$.
\item For $m = 0, \dots, p-1$, compute
\begin{equation*}
\widetilde{\mathbf y}_{(m+1)/p} := \widetilde{\mathbf y}_{m/p} + \frac{\Delta t}{p}
\sum_{\ell = 0}^{k-1} \beta_{m, \ell} \, {\mathbf w}_{n-\ell} +
\frac{\Delta t}{p} {\mathbf B}{\mathbf P} \sum_{\ell = 0}^{k-1} \alpha_{\ell} \, {\widetilde{\mathbf y}}_{(m-l)/p} \,.
\end{equation*}
\item Set ${\mathbf y}_{n+1} := \widetilde{\mathbf y}_1$.
\end{enumerate}}
Steps 1-4 correspond to the numerical solution of (\ref{eq:GrMi_mod_pr}) until $\tau = \Delta t$ with the $k$-step
AB scheme, using the local time-step $\Delta \tau = \Delta t / p$. For $\mathbf P = \mathbf 0$ or $p=1$, that is without
any local time-stepping, we thus recover the standard $k$-step Adams-Bashforth scheme. If the fraction of nonzero entries in
$\mathbf P$ is small, the overall cost is dominated by the computation of ${\mathbf w}_{n}$ in Step 3,  which requires one
multiplications by ${\mathbf B}({\mathbf I} - {\mathbf P})$ per time-step $\Delta t$. All further matrix-vector
multiplications by ${\mathbf B} {\mathbf P}$ only affect those unknowns that lie inside the refined region, or immediately
next to it; hence, their computational cost remains negligible as long as the locally refined region contains a small part of $\Omega$.

We have shown above how to derive LTS-AB$k$($p$) schemes of arbitrarily high accuracy. Since the third- and fourth-order LTS-AB$k$($p$)
schemes are probably the most relevant for applications, we now describe the LTS-AB$k$($p$) schemes for $k=3$, 4 and $p=2$.
Other examples of LTS Adams-Bashforth schemes are listed in \cite{GrMi_GM11}.

For $k=3$ and $p=2$, the LTS-AB$3$($2$) method reads:
\begin{equation*} \begin{split}
\widetilde{\mathbf y}_{1/2} = {{\mathbf y}}_n & + \frac{\Delta t}{2} {\mathbf B}({\mathbf I} - {\mathbf P}) \left [
\frac{17}{12} {\mathbf y}_n - \frac{7}{12} {\mathbf y}_{n-1} +  \frac{2}{12} {\mathbf y}_{n-2} \right ] \\
& + \frac{\Delta t}{2}{\mathbf B}{\mathbf P} \left [ \frac{23}{12} {{\mathbf y}}_n -
\frac{16}{12} {\widetilde{\mathbf y}}_{-1/2} + \frac{5}{12} {{\mathbf y}}_{n-1} \right ] \,, \\
{\mathbf y}_{n+1} = \widetilde{\mathbf y}_{1} = \widetilde{\mathbf y}_{1/2} & + \frac{\Delta t}{2} {\mathbf B}({\mathbf I} - {\mathbf
P}) \left [ \frac{29}{12} {\mathbf y}_n - \frac{25}{12} {\mathbf y}_{n-1} +  \frac{8}{12} {\mathbf y}_{n-2} \right ] \\
& + \frac{\Delta t}{2}{\mathbf B}{\mathbf P} \left [ \frac{23}{12}
{\widetilde{\mathbf y}}_{1/2} - \frac{16}{12} {{\mathbf y}}_{n} + \frac{5}{12} {\widetilde{\mathbf y}}_{-1/2} \right ] \,.
\end{split} \end{equation*}

For the case with $k=4$ and $p = 2$, we find the LTS-AB$4$($2$) scheme:
\begin{equation*} \begin{split}
\widetilde{\mathbf y}_{1/2}  = {{\mathbf y}}_n & +
\frac{\Delta t}{2} {\mathbf B}({\mathbf I} - {\mathbf P}) \left [
\frac{297}{192} {\mathbf y}_n - \frac{187}{192} {\mathbf y}_{n-1} +  \frac{107}{192} {\mathbf y}_{n-2} - \frac{25}{192} {\mathbf y}_{n-3}\right ] \\
& + \frac{\Delta t}{2}{\mathbf B}{\mathbf P} \left [ \frac{55}{24} {{\mathbf y}}_n
- \frac{59}{24} {\widetilde{\mathbf y}}_{-1/2} + \frac{37}{24} {{\mathbf y}}_{n-1} - \frac{9}{24} {\widetilde{\mathbf y}}_{-3/2} \right ] \,, \\
{\mathbf y}_{n+1} = \widetilde{\mathbf y}_{1} = \widetilde{\mathbf y}_{1/2} & + \frac{\Delta t}{2} {\mathbf B}({\mathbf I} - {\mathbf P}) \left [
\frac{583}{192} {\mathbf y}_n - \frac{757}{192} {\mathbf y}_{n-1} + \frac{485}{192} {\mathbf y}_{n-2} - \frac{119}{192} {\mathbf y}_{n-3} \right ] \\
& + \frac{\Delta t}{2}{\mathbf B}{\mathbf P} \left [ \frac{55}{24} {\widetilde{\mathbf y}}_{1/2}
- \frac{59}{24} {{\mathbf y}}_{n} + \frac{37}{24} {\widetilde{\mathbf y}}_{-1/2} - \frac{9}{24} {{\mathbf y}}_{n-1} \right ] \,.
\end{split} \end{equation*}

\begin{proposition}
The local time-stepping method LTS-AB$k$($p$) is consistent of order $k$.
\end{proposition}
\begin{proof} See \cite{GrMi_GM11}.\end{proof}
%%%%%%%%%%%%%%%%%%%%%%%%%%%%%%%%%%%%%%%%%%%%%%%%%%%%%%%%%%%%%%%%%%%%%%%%%%%%%%%%%%%%%%%%%%%%%%%%%%%%%%%%%%%%%%%%%%%%%%%%%

%%%%%%%%% SECTION 5 %%%%%%%%%%%%%%%%%%%%%%%%%%%%%%%%%%%%%%%%%%%%%%%%%%%%%%%%%%%%%%%%%%%%%%%%%%%%%%%%%%%%%%%%%%%%%%%%%%%%%
\section{Numerical results}
Here we present numerical experiments that validate the expected order of convergence of the above LTS methods and demonstrate their
usefulness in the presence of complex geometry. First, we consider a simple one-dimensional test problem illustrate the stability properties
of the different LTS schemes presented above and to show that they yield the expected overall rate of convergence when combined with a spatial
finite element discretization of comparable accuracy, independently of the number of local time-steps $p$ used in the fine region.
Then, we illustrate the versatility of our LTS schemes by simulating the propagation of a circular wave in a square cavity with a small sigma-shaped hole.

%%%%%%%%%%%%%%%%%%%%%%%%%%%%%%%%%%%%%%%%%%%%%%%%%%%%%%%%%%%%%%%%%%%%%%%%%%%%%%%%%%%%%%%%%%%%%%%%%%%%%%%%%%%%%%%%%%%%%%%%%
\subsection{Stability}
We consider the one-dimensional homogeneous damped wave equation (\ref{eq:GrMi_model_eq_1}) with constant wave speed $c = 1$ and
damping coefficient $\sigma = 0$ on the interval $\Omega = [0\, , \,6]$. Next, we divide $\Omega$ into three
equal parts.  The left and right intervals, $[0\, , \,2]$ and $[4\, , \,6]$, respectively, are discretized with
an equidistant mesh of size $h^{\mbox{\scriptsize coarse}}$, whereas the interval $\Omega_f = [2\, , \,4]$ is discretized with an
equidistant mesh of size $h^{\mbox{\scriptsize fine}} = h^{\mbox{\scriptsize coarse}} / p$. Hence, the two outer intervals
correspond to the coarse region whereas the inner interval $[2\, , \,4]$ to the refined region. In \cite{GrMi_DG09}, we have studied numerically the stability of
the LTS-LF2($p$) and the LTS-LFME4($p$) methods. To determine the range of values $\Delta t$ for which the LTS-LF2($p$) scheme is stable, the eigenvalues of
$(\Delta t^2 / 4) {\mathbf A}_p$ (${\mathbf A}_p$ is defined by (\ref{eq:GrMi_matrix_Ap})) for varying $\Delta t / \Delta t_{LF}$ are computed, where $\Delta t_{LF}$ denotes
the largest time-step allowed by the standard leap-frog method. The LTS-LF2($p$) scheme is stable for any particular $\Delta t$ if all corresponding eigenvalues
lie between zero and one; otherwise, it is unstable. We have observed that the largest time step allowed by the LTS-LF2($p$) scheme is only about 60\% of $\Delta t_{LF}$.
A slight extension (overlap) of the region where local time steps are used into that part of the mesh immediately adjacent to the refined region
typically improves the stability of the LTS-LF2($p$) scheme. Moreover, the numerical results suggested that an overlap by one element when combined with a
${\cal P}^1$ continuous FE discretization (with mass lumping), or by two elements when combined with a IP-DG  discretization, permits the use of the maximal (optimal)
time step $\Delta t_{LF}$. The numerical results also suggested that an overlap by one element for the IP-DG  discretization is needed for the
optimal CFL stability condition of the LTS-LFME4($p$) independently of $p$. Remarkably, no overlap is needed for the LTS-LFME4($p$) scheme to remain
stable with the optimal time-step when combined with the continuous ${\cal P}^3$ elements.

In \cite{GrMi_GM11}, we have considered the above one-dimensional problem with $\sigma = 0.1$. We have written the LTS-AB$k$($p$) scheme as a one-step method and
than studied numerically it stability when combined with a spatial finite element discretization of comparable accuracy. For a spatial discretization with
standard continuous, IP-DG or nodal DG finite elements, we have obtained that the maximal time-step $\Delta t_{p}$ allowed by the LTS-AB\,$2$($p$) scheme is
about 80 \% of the optimal time-step $\Delta t_{AB2}$ (the largest time-step allowed by the standard two-step AB method) independently of $h$, $p$ and $\sigma$;
moreover, the CFL stability condition of the LTS-AB\,$3$($p$) and LTS-AB\,$4$($p$) schemes is optimal for all $h$, $p$ and $\sigma$.

%%%%%%%%%%%%%%%%%%%%%%%%%%%%%%%%%%%%%%%%%%%%%%%%%%%%%%%%%%%%%%%%%%%%%%%%%%%%%%%%%%%%%%%%%%%%%%%%%%%%%%%%%%%%%%%%%%%%%%%%%
\subsection{Convergence}
We consider the one-dimensional homogeneous model problems (\ref{eq:GrMi_model_eq_1}) and (\ref{eq:GrMi_model_eq_2}) with constant wave speed $c = 1$ and
damping coefficient $\sigma = 0.1$ on the interval $\Omega = (0\, , \,6)$. The initial conditions are chosen to yield
the exact solution
\begin{equation*} \begin{split}
u(x, t) & = \frac{2 e^{-\frac{\sigma t}{2}}}{\sqrt{4\pi^2 - \sigma^2}} \sin(\pi x)
\sin\left ( \frac{t}{2} \sqrt{4\pi^2 - \sigma^2}\right ) \,, \\
v(x, t) & =  \frac{\partial u}{\partial t}(x, t) \,, \quad {\mathbf w} (x, t) = - \nabla u (x, t) \,.
\end{split} \end{equation*}
Again, we divide $\Omega$ into three equal parts. The left and right intervals, $[0, 2]$ and $[4, 6]$,
respectively, are discretized with an equidistant mesh of size $h^{\mbox{\scriptsize coarse}}$, whereas the interval
$[2, 4]$ is discretized with an equidistant mesh of size $h^{\mbox{\scriptsize fine}} = h^{\mbox{\scriptsize coarse}} / p$. Hence, the two outer
intervals correspond to the coarse region and the inner interval $[2\, , \,4]$ to the refined region.

\begin{figure}[t!]
\centering
\begin{tabular}{cc}
\includegraphics[scale=0.2]{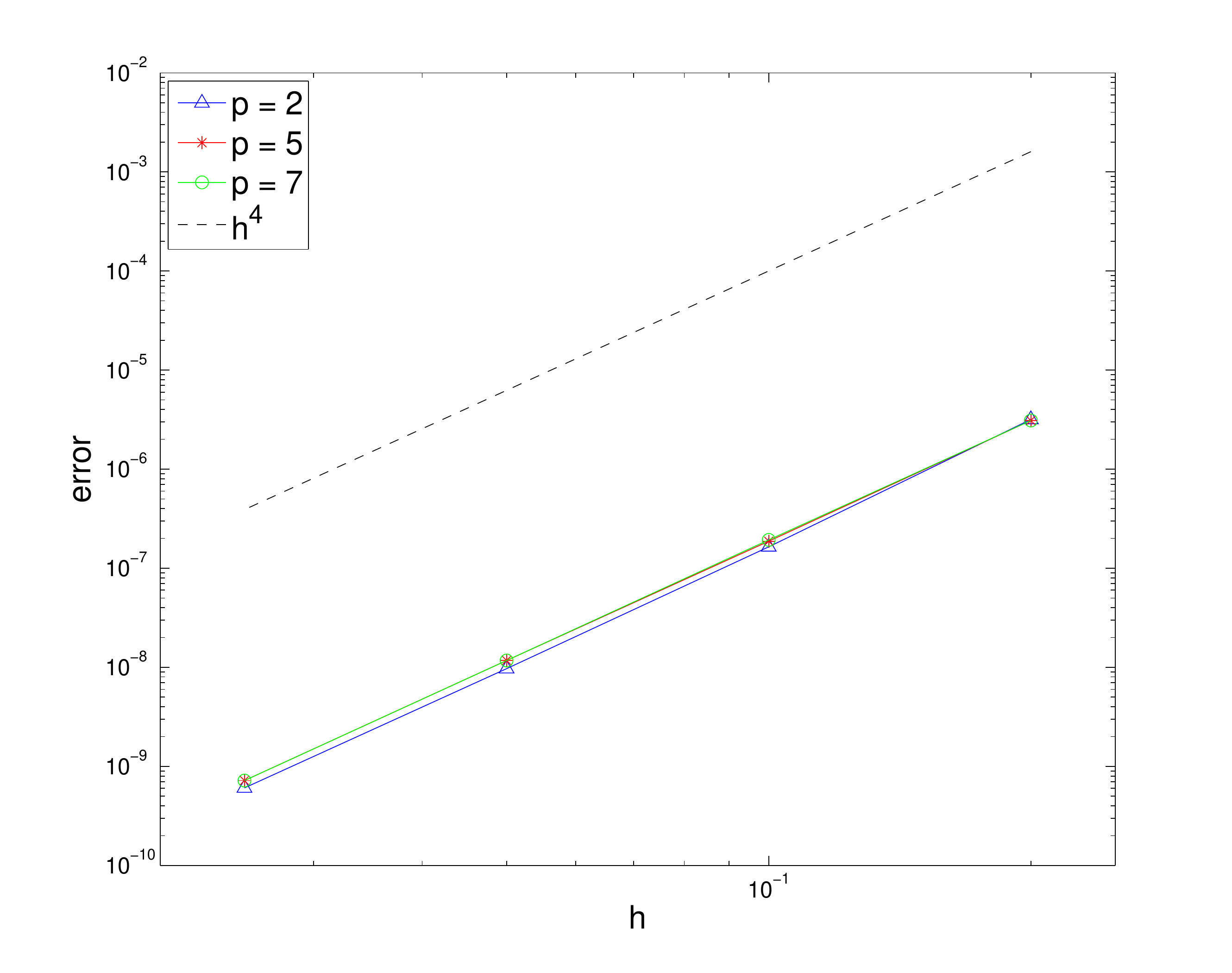}  &
\includegraphics[scale=0.2]{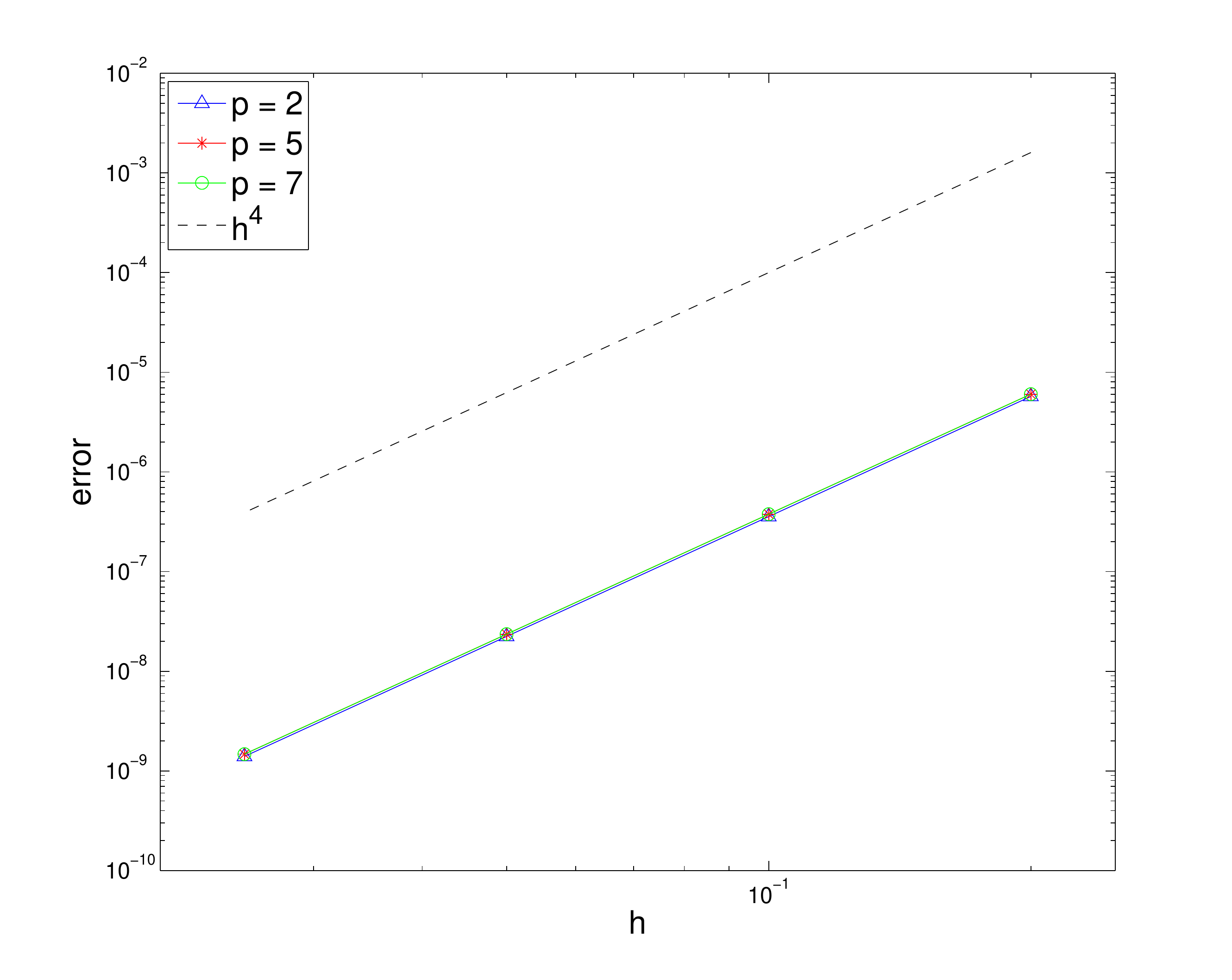}  \\
{\footnotesize (a) continuous FE ($h$ = 0.2, 0.1, 0.05, 0.025)} & {\footnotesize (b) IP-DG ($h$ = 0.2, 0.1, 0.05, 0.025)}
\end{tabular}
\begin{tabular}{c}
\includegraphics[scale=0.2]{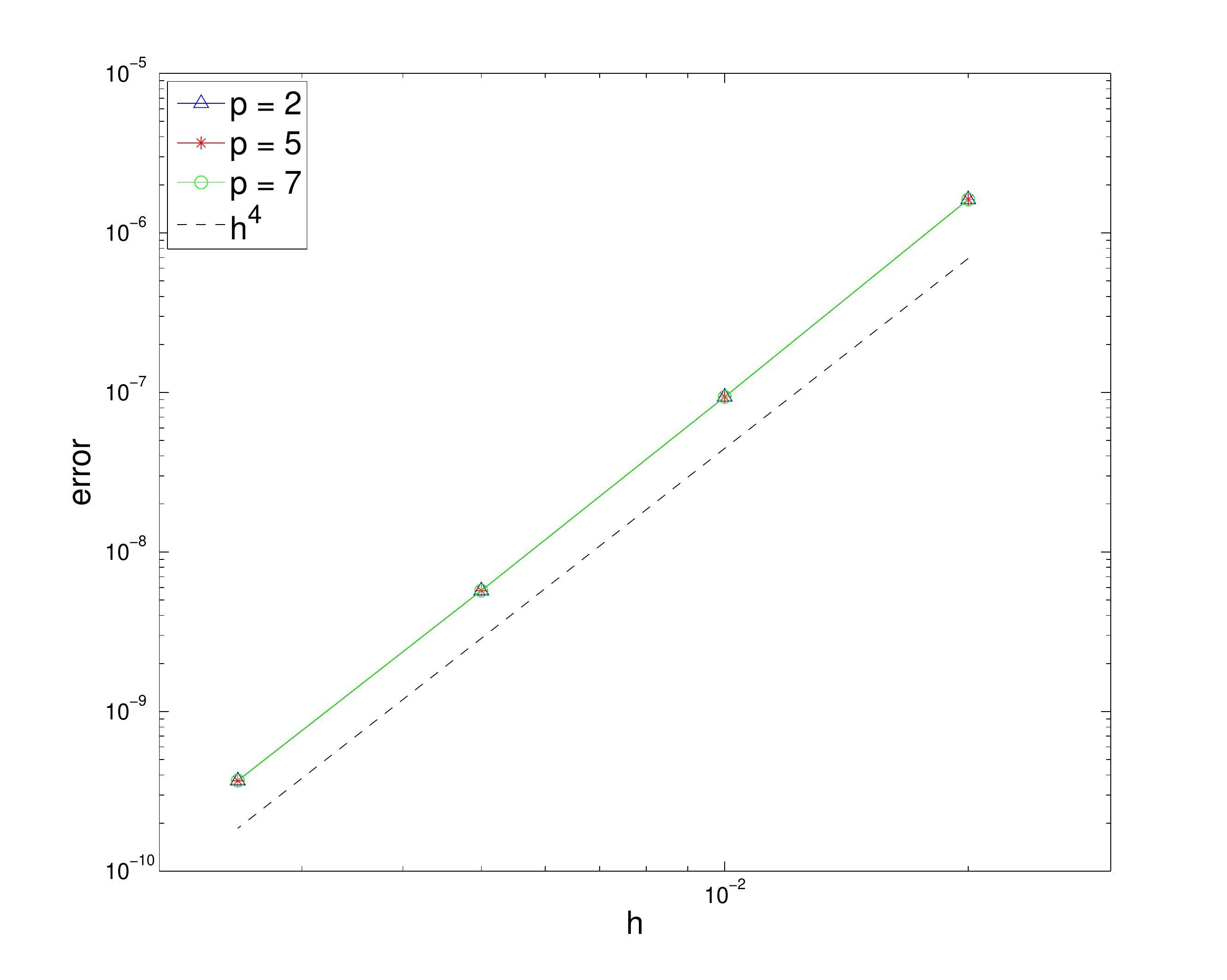} \\ {\footnotesize (c) nodal DG ($h$ = 0.02, 0.01, 0.005, 0.0025)}
\end{tabular}
\caption{LTS-AB$4$($p$) error vs.\ $h = h^{\mbox{\scriptsize coarse}}$ for ${\cal P}^3$ finite elements with $p=2, 5, 7$.} \label{fig:GrMi_1d_err_P3}
\end{figure}

First, we consider a ${\cal P}^3$ continuous FE discretization with mass lumping and a sequence of increasingly finer meshes. For every time-step $\Delta t$, we shall
take $p \geq 2$ local steps of size $\Delta \tau = \Delta t / p$ in the refined region, with the fourth-order time-stepping scheme LTS-AB$4$($p$).
The first three time-steps of each LTS-AB$4$($p$) scheme are initialized by using the exact solution. According to our results on stability, we set
$\Delta t = \Delta t_{AB4}$, the corresponding largest possible time-step allowed by the AB approach of order four on an equidistant mesh with
$h = h^{\mbox{\scriptsize coarse}}$. As we systematically reduce the global mesh size $h^{\mbox{\scriptsize coarse}}$, while simultaneously reducing $\Delta t$,
we monitor the $L^2$ space-time error in the numerical solution $\| u(\cdot, T) - u^h(\cdot, T)\|_{L^2(\Omega)}$ at the final time $T = 10$. In frame (a)
of Fig.~\ref{fig:GrMi_1d_err_P3}, the numerical error is shown vs.\ the mesh size $h = h^{\mbox{\scriptsize coarse}}$:
regardless of the number of local time-steps $p =2$, 5 or 7, the numerical method converges with order four.

We now repeat the same experiment with the IP-DG ($\alpha = 20$ in (\ref{eq:GrMi_param_alpha})) and the nodal DG discretizations with ${\cal P}^3$-elements.
As shown in frames (b) and (c) of Fig.~\ref{fig:GrMi_1d_err_P3}, the LTS-AB$4$($p$) method again yields overall fourth-order convergence independently of $p$.

\begin{remark}
We have obtained similar convergence results for other values of $p$ and $\sigma$. In summary, we observe the optimal rates convergence of order $k$ for the
LTS-AB$k$($p$) schemes as well as for the LTS-LF2($p$), LTS-LSME4($p$) and LTS-LFCN2($p$) schemes, regardless of the spatial FE discretization and independently
of the number of local time-steps $p$ and the damping coefficient $\sigma$. For more details, we refer to \cite{GrMi_DG09, GrMi_GM10, GrMi_GM11}.
\end{remark}

%%%%%%%%%%%%%%%%%%%%%%%%%%%%%%%%%%%%%%%%%%%%%%%%%%%%%%%%%%%%%%%%%%%%%%%%%%%%%%%%%%%%%%%%%%%%%%%%%%%%%%%%%%%%%%%%%%%%%%%%%
\subsection{Two-dimensional example}
To illustrate the usefulness of the LTS method presented above, we consider (\ref{eq:GrMi_model_eq_1}) in a square cavity $\Omega = (0, 1)^2$,
with a small sigma-shaped hole - see Figure~\ref{fig:GrMi_2d_domain}. We set the constant wave speed $c = 1$ and the damping coefficient
$$
\sigma(\mathbf{x}) = \left \{ \begin{array}{cl} 10\,, & x_2 < 0.5 \\0.1 \,, & \mbox{otherwise}\,. \end{array} \right.
$$
We impose homogeneous Neumann conditions on the boundary of $\Omega$ and choose as initial conditions
\begin{equation*} %\begin{split}
u_0({\mathbf x}) = 
\exp\left(-\Vert \mathbf{x}-\mathbf{x}_0 \Vert^2/r^2\right)\,,\qquad
v_0(\mathbf{x}) = 0 \,,
%\end{split}
\end{equation*}
where $\mathbf{x}_0=(0.45, 0.55)$ and $r=0.012$.

\begin{figure}[t!]
\centerline{\includegraphics[scale=0.5]{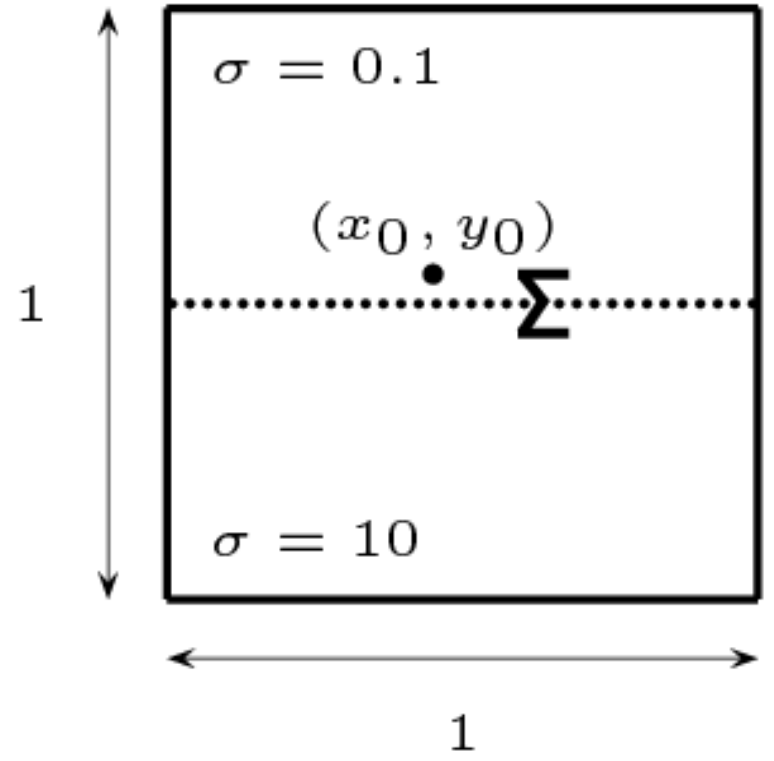}}
\caption{Two-dimensional example: the computational domain $\Omega$.} \label{fig:GrMi_2d_domain}
\end{figure}

For the spatial discretization we opt for the ${\cal P}^2$ continuous finite elements with mass lumping. First,
$\Omega$ is discretized with triangles of minimal size $h^{\mbox{\scriptsize coarse}}~=~0.03$. However, such triangles
do not resolve the small geometric features of the sigma-shaped hole, which require
$h^{\mbox{\scriptsize fine}} \approx h^{\mbox{\scriptsize coarse}} / 7$, as shown in Figure \ref{fig:GrMi_2d_mesh}.
Then, we successively refine the entire mesh three times, each time splitting every triangle into four.
Since the initial mesh in $\Omega$ is unstructured, the boundary between the fine and coarse mesh is not well-defined.
Given $h^{\mbox{\scriptsize coarse}}$, here the fine mesh corresponds to all triangles with
$h < 0.75 h^{\mbox{\scriptsize coarse}}$ in size, that is the darker triangles in Figure \ref{fig:GrMi_2d_mesh}. The
corresponding degrees of freedom in the finite element solution are then selected merely by setting to one the
corresponding diagonal entries of the matrix ${\mathbf P}$.

\begin{figure}[t!]
\begin{center}
\centerline{\includegraphics[scale=1]{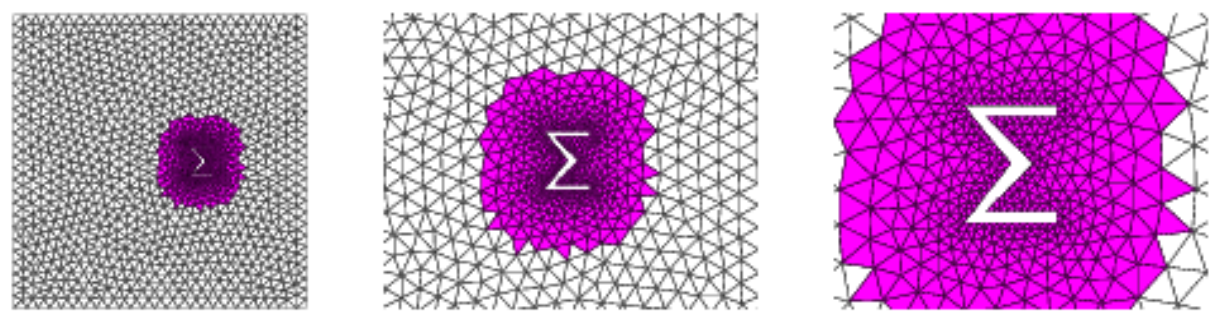}}
\end{center}
\caption{The triangular initial mesh at various magnification rates: the darker triangles belong to the ``fine'' mesh.} \label{fig:GrMi_2d_mesh}
\end{figure}

\begin{figure}[t!]
\centerline{\includegraphics[scale=0.8]{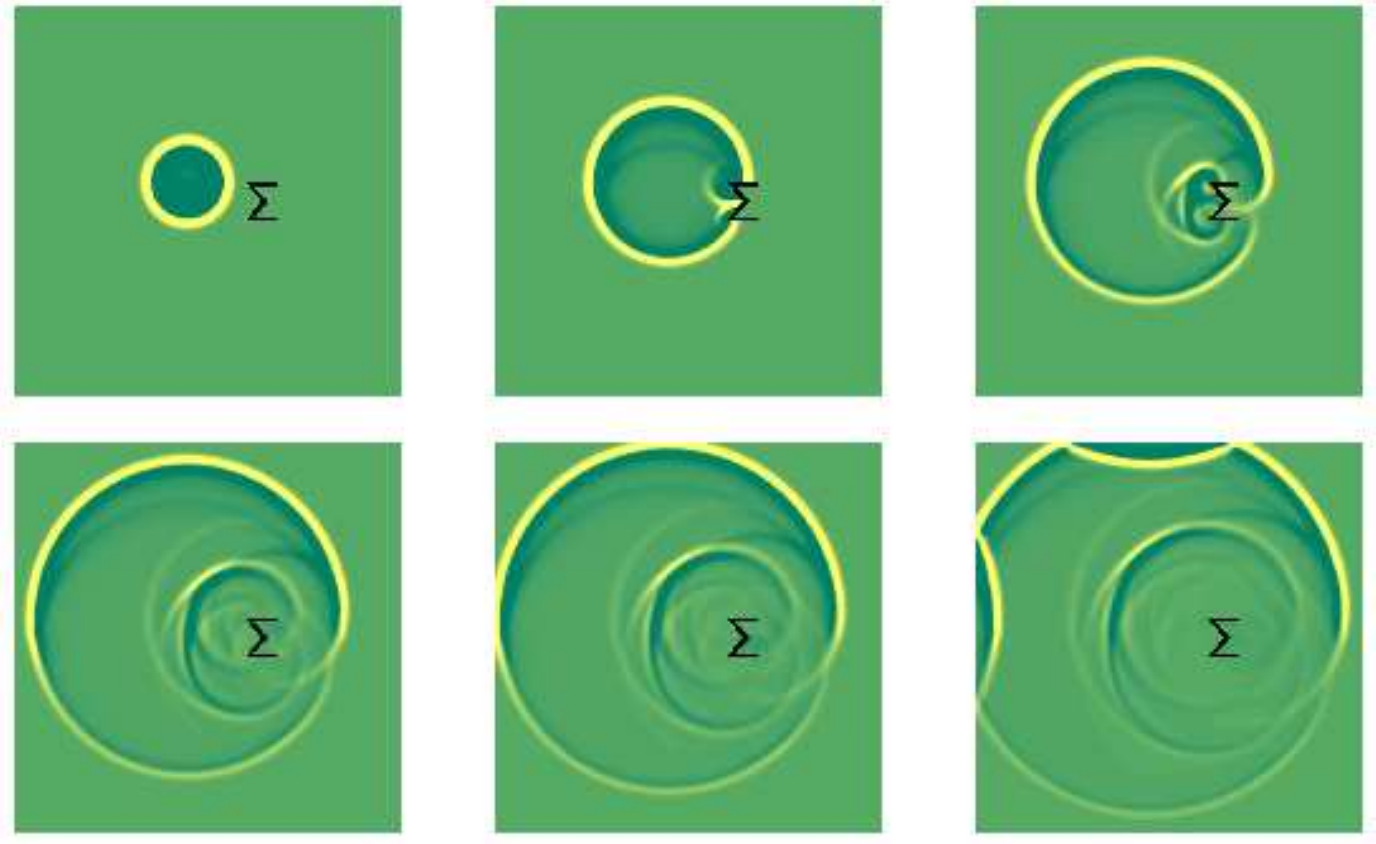}}
\caption{Gaussian pulse penetrating a cavity with a small sigma-shaped hole.
The solution is shown at times $t=0.1$,  $0.2$, $0.3$, $0.4$, $0.44$ and $0.5$.} \label{fig:GrMi_2d_sol}
\end{figure}

For the time discretization, we choose the third-order LTS-AB$3$($7$) time-stepping scheme with $p = 7$, which for every time-step $\Delta t$ takes seven local
time-steps inside the refined region. Thus, the numerical method is third-order accurate in both space
and time under the CFL condition $\Delta t = 0.07 \, h^{\mbox{\scriptsize coarse}}$, determined experimentally. If instead the same (global) time-step
$\Delta t$ was used everywhere inside $\Omega$, it would have to be about seven times smaller than necessary in most of $\Omega$. As a starting procedure,
we employ a standard fourth-order Runge-Kutta scheme.

In Fig.~\ref{fig:GrMi_2d_sol}, snapshots of the numerical solution are shown at different times. The circular wave, initiated by the Gaussian pulse,
propagates outward until it impinges first on the sigma-shaped hole and later on the upper and left boundaries of $\Omega$. The reflected waves move back into
$\Omega$ while multiple reflections occur both at the obstacle and along the interface at $x_2 = 0.5$. As the waves cross that interface and penetrate the
lower part of $\Omega$, they are strongly damped.

%\begin{figure}[t!]
%\centerline{\includegraphics[scale=0.8]{GrMi_2d_sol.eps}}
%\caption{Gaussian pulse penetrating a cavity with a small sigma-shaped hole.
%The solution is shown at times $t=0.1$,  $0.2$, $0.3$, $0.4$, $0.44$ and $0.5$.} \label{fig:GrMi_2d_sol}
%\end{figure}

%%%%%%%%%%%%%%%%%%%%%%%%%%%%%%%%%%%%%%%%%%%%%%%%%%%%%%%%%%%%%%%%%%%%%%%%%%%%%%%%%%%%%%%%%%%%%%%%%%%%%%%%%%%%%%%%%%%%%%%%%

%%%%%%%%% SECTION 6 %%%%%%%%%%%%%%%%%%%%%%%%%%%%%%%%%%%%%%%%%%%%%%%%%%%%%%%%%%%%%%%%%%%%%%%%%%%%%%%%%%%%%%%%%%%%%%%%%%%%%
\section{Concluding remarks}
%%%%%%%%%%%%%%%%%%%%%%%%%%%%%%%%%%%%%%%%%%%%%%%%%%%%%%%%%%%%%%%%%%%%%%%%%%%%%%%%%%%%%%%%%%%%%%%%%%%%%%%%%%%%%%%%%%%%%%%%%
Starting from the classical leap-frog (LF) or Adams-Bashforth (AB) methods, 
we have presented explicit local time-stepping (LTS) schemes 
for wave equations, either with or without damping. By allowing
arbitrarily small time-steps precisely where the smallest elements 
in the mesh are located, these LTS schemes circumvent the crippling effect of locally
refined meshes on explicit time integration. 

When combined with a spatial finite element discretization with an essentially diagonal mass matrix, 
the resulting LTS schemes remain fully explicit. Here three such finite element discretizations
were considered: standard $H^1$-conforming finite elements with mass-lumping, 
an IP-DG formulation, and nodal DG finite elements.
In all cases, our numerical results demonstrate that the resulting fully discrete numerical
schemes yield the expected space-time optimal convergence rates. Moreover, the LTS-AB$(k)$
schemes of order $k \geq 3$ have optimal CFL stability properties regardless of the mesh size $h$, 
the global to local step-size ratio $p$, or the dissipation $\sigma$. Otherwise, the CFL condition 
of the LTS scheme may be sub-optimal; then, 
by including a small overlap of the fine and the coarse region, the CFL condition of the resulting
LTS scheme can be significantly enhanced.

Since the LTS methods presented here are truly explicit, their parallel implementation is straightforward. Let $\Delta t$ denote the 
time-step imposed by the CFL condition in the coarser part of the mesh. Then, during every (global) time-step $\Delta t$, each local time-step of size 
$\Delta t / p$ inside the fine region of the mesh simply corresponds to sparse matrix-vector multiplications that only involve degrees of freedom
associated with the fine region of the mesh. Those ``fine'' degrees of freedom can be selected individually and without any restriction by setting the 
corresponding entries in the diagonal projection matrix $P$ to one; in particular, no adjacency or coherence in the numbering of the degrees of freedom is assumed. 
Hence the implementation is straightforward and requires no special data structures.  

In the presence of multi-level mesh refinement, each local time-step in the fine
region can itself include further local time-steps inside a smaller subregion with an even higher degree of local mesh refinement.
The explicit local time-stepping schemes developed here for the scalar damped wave 
equation immediately apply to other 
damped wave equations, such as in electromagnetics 
or elasticity; in fact, they can be used for general linear first-order hyperbolic systems.

%%%%%%%%%% BIBLIOGRAPHY %%%%%%%%%%%%%%%%%%%%%%%%%%%%%%%%%%%%%%%%%%%%%%%%%%%%%%%%%%%%%%%%%%%%%%%%%%%%%%%%%%%%%%%%%%%%%%%%%

\end{document}